\documentclass{elsarticle}

\newcommand{\be}{\begin{eqnarray}}
\newcommand{\ee}{\end{eqnarray}}
\newcommand{\ben}{\begin{eqnarray*}}
\newcommand{\een}{\end{eqnarray*}}

\newcommand{\Gv}{G\! -\! v}

\newcommand{\ibar}{\bar{i}}

\newtheorem{theorem}{Theorem}
\newtheorem{lemma}[theorem]{Lemma}

\newenvironment{proof}[1][Proof]{\begin{trivlist}\item[\hskip \labelsep {\bfseries #1}]}{\end{trivlist}}

\newenvironment{definition}[1][Definition]{\begin{trivlist}\item[\hskip \labelsep {\bfseries #1}]}{\end{trivlist}}

\newcommand{\smallvertices}[1]{ \foreach \pos/\name in {#1} \node[smallvertex] (\name) at \pos {}; }

\newcommand{\edges}[1]{ \foreach \source/ \dest in {#1} \path[edge] (\source) -- (\dest); }

\newcommand{\basicgraph}[3]{ \tikz[scale=#1]{\smallvertices{#2} \edges{#3}} }

\usepackage{graphicx}
\usepackage{amssymb}

\usepackage{pgfplots}
\usetikzlibrary{pgfplots.groupplots}
\usetikzlibrary{shapes,arrows}

\begin{document}

\tikzstyle{vertex}=[circle, draw, inner sep=0pt, minimum size=5pt]
\tikzstyle{smallvertex}=[circle, fill=black, draw, inner sep=0pt, minimum size=4pt]
\tikzstyle{fillvertex}=[circle, fill=black, draw, inner sep=0pt, minimum size=5pt]
\tikzstyle{dot}=[circle, draw, fill=black, inner sep=0pt, minimum size=2pt]
\tikzstyle{doublevertex}=[circle, draw, inner sep=0pt, minimum size=5pt, double distance=1pt]
\tikzstyle{stub}=[circle,  inner sep=0pt, minimum size=1pt]
\tikzstyle{bigvertex}=[circle, draw, inner sep=0pt, minimum size=12pt]
\tikzstyle{dmvertex}=[diamond, draw, fill, inner sep=0pt, minimum size=4pt]
\tikzstyle{edge} = [draw,-]
\tikzstyle{thickedge} = [draw,thick,-]
\tikzstyle{dotedge} = [draw,dotted]
\tikzstyle{dashedge} = [draw,dashed]

\title{Graph reconstruction and generation from one card and the degree sequence}

\author{Andrew M. Steane}
\ead{a.steane@physics.ox.ac.uk}
\address{
Department of Atomic and Laser Physics, Clarendon Laboratory,\\
Parks Road, Oxford, OX1 3PU, England.
}

\date{\today}

\begin{abstract}
Many degree sequences can only be realised in graphs that contain a `ds-completable card',
defined as a vertex-deleted subgraph in which the erstwhile neighbours of the deleted vertex
can be identified from their degrees, if one knows the degree sequence of the original graph.
We obtain conditions on the degree sequence, such that graphs whose degree sequence satisfies
one of the conditions must contain such a card. The methods 
allow all such sequences on graphs of order up to 10 to be identified,
and some fraction of the sequences for larger graphs. Among other
applications, this can be used to reduce the computational task of generating 
graphs of a given degree sequence without duplicates.
\end{abstract}


\maketitle

If any single vertex is deleted from a regular simple graph $G$ having more than 2 vertices, 
then one obtains
a graph $G'$ with the following interesting property: there is one, and only one,
regular graph (up to isomorphism) having $G'$ as an induced sub-graph. (Proof: the graph
can be uniquely constructed from $G'$ as follows: 
if $G'$ is empty then add to it a further isolated vertex, otherwise
add to $G'$ a further vertex with edges to all vertices
in $G'$ having the smallest degree in $G'$.) In this paper we investigate a more general
property that is related to this. It can happen that, for a given graphic sequence $\pi$ (i.e. a
non-decreasing integer sequence which is the degree sequence of at least one simple graph), 
any graph having that degree sequence contains at least one sub-graph $G'$ such that the original
graph can be uniquely obtained from $G'$ and $\pi$. One might say that such a sub-graph
is uniquely `completable', given $\pi$, or that the original graph $G$ is 
uniquely reconstructible from $G'$ and $\pi$.

Such graphs are interesting in the context of the {\em reconstruction conjecture} of
(Kelly and Ulam 1942), and in the context of graph generation. The former is one of the most 
well-known unsolved problems in graph theory. It concerns the question, whether or
not the collection of vertex-deleted subgraphs of any given graph uniquely specify the
graph (a precise statement is given below). 
Graph generation concerns the construction of graphs with some specified properties, where
typically one wishes to obtain an exhaustive list without duplicates.

The notion of reconstructing a graph after deletion of a vertex can also be seen as a form
of error-correction, or recovery of information from noisy data.

Let $G-v$ be the vertex-deleted subgraph of a simple graph $G$, obtained by deleting vertex $v$.
For given $G$, is there a $v$ such that it possible to deduce $G$ from $G-v$ if the degree 
sequence of $G$ is known? For most graphs, the answer is no, but the class of graphs
for which the answer is yes is large: it contains approximately half the graphs on up to 9 vertices,
and 37\% of the graphs on 10 vertices, for example. 
In many cases, the degree sequence itself forces the graph to have this property, in the
sense that all graphs realising the given degree sequence contain such a subgraph.
We would like to identify these degree sequences.

For any graph, the degree sequence is itself recoverable from the deck.
The notion of augmenting a card by degree-related information was considered by Ramanachandran
\cite{Raman81}; see also \cite{Barrus10}. These authors were interested in reconstruction
from several cards, where for each card the degree of the deleted vertex is given. The present
work is concerned with reconstruction of a graph from a single card, when the whole
degree sequence of the graph is given. The methods to be discussed also bear on related
issues, such as the question, for any given graph of order $n$, in how many ways 
may a vertex be added so as to obtain a graph of order $n+1$ with a given `target' degree sequence.

\section{The reconstruction conjecture and terminology}

A graph $G = [V,E]$ is a set of vertices $V$ and edges $E$. We restrict attention to simple graphs
(those which are undirected and with no self-loops or multiple edges).
The {\em degree sequence} of $G$
is the multiset of degrees of the vertices, in non-descending order. A vertex of degree 1 is called
a {\em pendant}. 
The set of neighbours of a vertex $v$ will be denoted $N(v)$. The set $N(v) \cup \{v\}$
(called the closed neighbourhood of $v$) will be denoted $N[v]$.

For any graph $G$, a vertex-deleted subgraph $\Gv$ is called a {\em card}, and the multiset of
cards associated with all the vertices of $G$ is called the {\em reconstruction deck} of $G$,
or {\em deck} for short. 
Two cards are deemed `the same' if they are isomorphic and distinct if they are not. 
A property of $G$ which can be obtained from the deck, without
direct access to $G$, is said to be {\em recoverable}. For example, the number of vertices 
$n = |V(G)|$
is recoverable since it is equal to the number of cards. The number of edges 
$m = |E(G)|$ is recoverable since the number of edges appearing on all the cards is
$(n-2)m$ (since each edge appears on all but two of the cards). Further recoverable information
includes the degree sequence of $G$, and
for any given card $\Gv$, the degree sequence, in $G$,
of the neighbours of the associated vertex $v$.

A recoverable binary property, such as membership of a class, is said to be
{\em recognisable}. For example, the condition that the graph is regular is recognizable,
since it can be determined from a recoverable property, namely the degree sequence.

The question arises, whether a graph can be reconstructed in a unique way
from its deck. The graphs on two vertices cannot (they both have
the same deck). An automated (i.e. computer) calculation has confirmed that
graphs on 3 to 11 vertices can \cite{97McKay}, and for larger $n$ the question is open
apart from some special classes of graphs (see \cite{Bondy91,Babai95,14Bartha} for a review).
The {\em Reconstruction Conjecture} (Kelly and Ulam 1942) is the conjecture
that every graph having more than two vertices can be reconstructed in a unique
way (up to isomorphism) from its deck \cite{Kelly42,Kelly57,Ulam60}. 
The conjecture is known to be true for most graphs
(that is, of all graphs of given order, most are reconstructible) \cite{Bollobas90,McMullen07}. However, despite
the wealth of information about a graph that is provided by its deck, no
one has yet found out how to do the detective work and show that
the evidence to be found in each deck points to one and only one  `culprit' graph.
(It is not sufficient to find, for any given deck, a graph which has that deck.
Rather, one must show that no other graph has that same deck.)

A standard way to proceed in graph reconstruction problems is by a combination of
`recognise and recover'. One finds a reconstruction method which only applies to some
class of graphs, and one shows that membership of the class is recognisable.

\section{ds-reconstructibility}

It can happen that a reconstruction method may apply to all graphs whose degree
sequence has some given property. Such methods are useful because the degree
sequence is very easy to recover from the deck, or from any graph whose reconstructibility
is in question, and also because the degree sequence can be used as a starting-point for
generating sets of graphs. This connection can be used either to make
graph generation more efficient (see section \ref{s.generate}), or to avoid it
altogether when one is interested in the reconstruction problem.
In the latter case, if one wanted to use a computer to test for reconstructibility
all graphs of some given property, one might proceed by generating, for each degree sequence,
all graphs of that degree sequence and having the property, and testing each. If it is
already known that all
graphs of a given degree sequence are reconstructible, than one can avoid parts of
such a search. 

Before approaching the issue in general, first let us note some well-known and simple cases:

\begin{lemma}  \label{simple}
A graph with any of the following properties is reconstructible.\
\begin{enumerate}
\item There is an isolated vertex.
\item There is a full vertex (one of degree $d=n-1$).
\item  No pair of degrees differs by 1.
\item  There is a vertex of degree $d_v$ such that, after that degree is removed from
the degree sequence, the resulting sequence is one in which no pair of degrees differs by 1
(for example, the sequence 
$2, 2, 2, 4, 4, 4, 5, 6, 6, 9$ has this property).
\end{enumerate}
\end{lemma}

\begin{proof}
These examples are straightforward. In each case, there exists a card
$\Gv$ 
for which all the neighbours of $v$ are recognisable by their degree (using the fact that the
degree sequence of the neighbours of $v$ is recoverable, and their degrees in $\Gv$ are
one less than their degrees in $G$). 
In case 3, this is true of any card, and in case 4 it is true of the card associated
with the special vertex.   \qed  \end{proof}

{\bf Corollary}. Regular graphs are reconstructible; Euler graphs are reconstructible.

This follows immediately. The former was noted by Kelley 1957 and
Nash-Williams \cite{78NashWilliams}.

We would now like to generalize from these simple cases. That is, we want to define some
reasonably natural class of graphs that can be reconstructed by using degree information,
and then enquire how members of the class can be recognised.

To this end, let us introduce the following concept:

\begin{definition}
A card $\Gv$ is {\em ds-completable} if, for each degree $d$,  vertices of degree $d$ in the card
are either all neighbours or all non-neighbours of $v$ in $G$. 

We shall also say that a vertex $v$ is ds-completable in $G$
if the card $\Gv$ is ds-completable.
\end{definition}

For example, all the cards of a regular graph are ds-completable, and in the fourth case of lemma
\ref{simple} the card $\Gv$ is ds-completable, for the vertex $v$ defined there.

The idea of this definition is that, if a card has this property, then the set of neighbours of
the vertex which has been deleted
can be determined from the degree sequences $ds(G)$ and $ds(\Gv)$ without having to
bring in further information about the structure of the graph. For, we already know that the
degree sequence of the neighbours of $v$ is recoverable, so, for any given card $\Gv$, we know
what degrees the erstwhile neighbours of $v$ must have in that card (namely, one less than their
degrees in $G$). If the card has the property given in the definition, then there is no ambiguity in
determining, from the degree information alone, which vertices are the neighbours we want to find
(see the detailed example below).
Having identified them, the graph is reconstructed by introducing the missing vertex and attaching
it to those neighbours.

An automated method to identify the neighbours of $v$ in $\Gv$ is as follows \cite{14Bartha}.
First write out the degree sequence of the graph, and then, underneath it, write out the degree
sequence of $\Gv$, with a blank inserted at the first location where $d_v$ would appear.
For example:
\[
\begin{array}{lcccccccccccc}
G:&&2&2&3&3&3&5&5&6&7&7&9 \\
\Gv:&&1&1&2&2&3&4&4&5&-&7&9
\end{array}
\]
Next, insert vertical bars so as to gather the groups of given degree in the second line:
\[
\begin{array}{lccc|cc|c|cc|c|c|c|c}
G:&&2&2&3&3&3&5&5&6&7&7&9 \\
\Gv:&&1&1&2&2&3&4&4&5&-&7&9
\end{array}
\]
Then one inspects the groups. Since the degree sequence of the neighbours is recoverable, we
know which degrees in the lower line are candidates to belong to neighbours of $v$. If for a given
candidate group in the lower line, all the items in the top line for that group have the same
degree $d$, then all the vertices of degree $d-1$ in $\Gv$ are neighbours of $v$. If this
happens for all such groups, then no ambiguity arises, so the card can be completed.
The example shown above has this property. It is owing to the fact that, 
whereas $v$ has neighbours of degree 3,6,7, which might in this example lead to ambiguity,
it is also adjacent to all the vertices of degree 2,5,6, which removes the ambiguity.

If, instead, the neighbours of $v$ had degrees $2,3,3,5,6,7,9$, then the procedure would give
this table:
\[
\begin{array}{lcc|ccc|c|c|cc|c|c|c}
G:&&2&2&3&3&3&5&5&6&7&7&9 \\
\Gv:&&1&2&2&2&3&4&5&5&6&-&8
\end{array}
\]
Now there are two groups (those for $d=2$ and $d=5$ in $\Gv$) where the items in the upper
line are not all the same, and therefore we have two sets of vertices in $\Gv$ where we can't tell,
from this information alone, which are neighbours and which non-neighbours of $v$.

With this concept in mind, we can define an interesting class of graphs:

\begin{definition}
A graph $G$ is {\em ds-reconstructible} iff its deck contains a ds-completable card.
\end{definition}

\begin{definition}
A degree sequence is {\em ds-reconstruction-forcing} (or {\em forcing} for short)
iff it is not possible to construct a non-ds-reconstructible graph realizing that degree
sequence.
\end{definition}

(Equivalently, a graphic sequence is forcing iff every graph realizing it is 
ds-reconstructible, and non-graphic sequences are deemed to be forcing.)

A forcing degree sequence requires or forces the presence of a ds-completable card in any
associated graph. The use of the term `forcing' here is similar to, but not identical to, its use in \cite{Barrus08}.
In \cite{Barrus08} the authors use the term `degree-sequence-forcing' to describe a class of graphs.
A class is said to be `degree-sequence-forcing' if whenever any realization of a graphic 
sequence $\pi$ is in the class, then every other realization of $\pi$ is also in the class. In the present
work we use the term `forcing' to describe a class of degree sequences. The set of
graphic sequences is divided into three distinct classes: those for which every realisation is
ds-reconstructible, those for which no realization is ds-reconstructible, and those for which some
realizations are ds-reconstructible and some are not. The first of these classes is under study here.

If we can show that a given degree sequence is ds-reconstruction forcing, then we 
shall have proved the reconstructibility of all graphs realising that degree sequence.
We shall also have proved a stronger result, namely that for every graph $G$ realising the
forcing degree sequence $\pi$, there is a graph $G'$ on one fewer vertices,
such that $G$ is the only graph realising $\pi$ that can be obtained by adding a vertex
and edges to $G'$. This is relevant to graph generation, which we discuss in section \ref{s.generate}.

It is the main purpose of this paper to present methods to determine whether or not
a given degree sequence is ds-reconstruction-forcing.
Even if the reconstruction conjecture were to be proved by other methods, it would still remain
an interesting question, which degree sequences are ds-reconstruction-forcing, and whether
or not they can be easily identified. Connections to closely related ideas in graph theory are
explored briefly in the conclusion.

Our task is in some respects comparable to the task of determining whether or not a given 
integer sequence is graphic (i.e., is the degree sequence of some graph). It is more closely
comparable to the task of determining whether a graphic sequence is
unigraphic.\cite{Johnson80,Tyshkevich00,Kleitman75,Koren76}
That question has been largely solved, in that there exists a small collection of conditions on
the degree sequence that can be evaluated in linear time and which suffice to establish
whether or not a sequence is unigraphic \cite{Kleitman75,Koren76}. Another method is
to construct a realization of a given sequence and apply switching operations to the
graph \cite{Johnson80,Tyshkevich00}.

The {\em switch}, also called {\em transfer}, is an operation which we shall write $\otimes^{ac}_{bd}$,
where $a,b,c,d$ are four distinct vertices of a graph $G$. If edges $ac$, $bd$ are present in the graph,
and $ad$, $bc$ are not, then $G$ is said to {\em admit} the switch $\otimes^{ac}_{bd}$. The image
of $G$ under $\otimes^{ac}_{bd}$ is obtained by replacing edges $ac$, $bd$ by $ad$, $bc$; i.e.
$\otimes^{ac}_{bd} G = G - ac - bd + ad + bc$. Such an operation preserves the degrees of all
the vertices, and furthermore it may be shown that one can move among all graphs of given degree
sequence via a finite sequence of switches \cite{Fulkerson65}.

\begin{lemma}   \label{lemswtich}
If $\Gv$ is ds-completable, then $(\otimes^{ac}_{bd}G) - v$ is ds-completable when 
$v \notin \{a,b,c,d\}$.
\end{lemma}
{\em Proof.} Such an operation does not change the neighbourhood of $v$, nor the degree of
any vertex, hence if $v$ is completable in $G$, it is also completable in $\otimes^{ac}_{bd}G$. \qed

One way to investigate whether a sequence $\pi$ is ds-reconstruction forcing is to 
construct a graph $G$ realizing $\pi$, and then, if $G$ is ds-reconstructible, perform
a switch which changes the neighbourhood of the first ds-completable vertex found. 
Such a switch may possibly result in a graph with fewer ds-completable vertices. One continues
switching until a graph is found having no completable vertex, or until the search is
abandoned. If a non-ds-reconstructible graph was found, or if all graphs realizing $\pi$ have
been considered, then one has discovered whether or not $\pi$ is forcing;
otherwise one has an inconclusive result. 

When the number of realizations of a given sequence $\pi$
is small, one can move among them (and also be sure that all cases have been considered), by
exhaustive implementation of switching, or one can construct them all another way. 
In this way the question, whether or not $\pi$ is forcing, can be decided. 
However such an exhaustive method is
prohibitively slow for most sequences. It would be useful to find a method to choose a sequence
of switches
that efficiently arrives at a non-ds-completable $G$ if one exists. I have not discovered
such a method. Instead the present work is devoted to finding conditions on the degree sequence
that must be satisfied by forcing sequences. The following ideas will be useful.

\begin{definition}
A degree $d$ is {\em bad} if the graph contains a vertex of degree $d - 1$.
A degree $d$ is {\em dull} if the graph contains a vertex of degree $d + 1$.
A vertex is bad or dull (or both), according as its degree is bad or dull (or both).
A vertex which is not bad is good. A set of vertices is said to be
good, bad or dull iff all its members are good, bad or dull, respectively.
\end{definition}
Example: in the following degree sequence, the bad degrees are underlined and the dull
degrees are overlined: 
$\overline{2,2},\underline{3,3,3}, 5,5,\overline{7,7}, \overline{\underline{8}}, \underline{9}$.

For every bad vertex, there is at least one dull vertex, and for every dull vertex,
there is at least one bad vertex. Vertices of least degree are never bad; vertices of highest
degree are never dull. If a vertex is bad in $G$ then it is dull in
the complement graph $\bar{G}$, and {\em vice versa}.

Bad neighbours make the job
of reconstruction harder. For, if a vertex $v$ has no bad neighbours, then, in the
associated card $\Gv$, the neighbours of $v$ are immediately recognisable since they have degrees
that do not occur anywhere in the degree sequence of $G$, and this is not true of
any other vertices in $\Gv$. Consequently the card is ds-completable.
If, on the other hand, a vertex has one or more bad neighbours then
this argument does not apply. This does not necessarily mean its card is not ds-completable, however:

\begin{theorem} \label{dsc}   [Mulla 1978]
The card $\Gv$ is ds-completable iff in $G$, no non-neighbour of $v$ (excluding $v$ itself)
has a degree one less than a neighbour of $v$.
\end{theorem}

Another way of stating the theorem is as follows.
{\em For any given vertex $v$ in graph $G$, let $D_v$ be the set of vertices of degree one less than a
neighbour of $v$, and let $N[v]$ be the closed neighbourhood of $v$ in $G$. Then
the card $\Gv$ is ds-completable iff $D_v \subset N[v]$.}

\begin{proof}
This has been proved by Mulla \cite{78Mulla} (see also \cite{14Bartha})
but we present a proof suited to the concepts employed in this paper. The idea is to work along
the degree sequence of the card $\Gv$, determining which degrees must belong to neighbours of $v$.
The only neighbours which are not immediately identifiable are the bad neighbours,
and the condition of the theorem renders those identifiable too. This is because if
all vertices of some degree $d$ in $G$ are either absent from the card $\Gv$
or have had their degree reduced by one, then we know that vertices in the card of degree $d$
must be neighbours of $v$. Conversely, if this is not the case, then there remains an ambiguity
and the neighbours cannot be determined by the degree information alone. \qed  \end{proof}

Note that we are not trying to capture cases of graphs which can be reconstructed 
by methods that involve further considerations in addition to degree information. 
For example, consider the path graph $P_4$. If one of the pendants is
deleted, the resulting card is {\em not} ds-completable (it does not satisfy the definition) because
there is an ambiguity: there are two vertices whose degree makes them candidates 
to be the neighbour of the pendant which was deleted. The fact that in this example
one gets the same graph, irrespective of which of these two is picked, is not under
consideration in the definitions employed here, but we will return to this point at the end.

\begin{lemma}  \label{lem.comp}
A graph $G$ is ds-reconstructible iff its complement $\bar{G}$ is
ds-reconstructible.
\end{lemma}
{\bf Corollary} A degree sequence is forcing iff the complement degree sequence is
forcing, where for any degree sequence $[ d_i ], \, i=1 \cdots n$, the complement degree
sequence is defined to be $[ n-1-d_n,\, n-1-d_{n-1},\, \cdots, n-1-d_1 ]$.

\begin{proof} By definition, a graph $G$ is ds-reconstructible iff it
contains a ds-completable vertex. Now, for any $G$,
a vertex which is bad in $G$ is dull in $\bar{G}$ and a vertex which is dull in $G$
is bad in $\bar{G}$. From lemma (\ref{dsc}) we have that a vertex $v$ is ds-completable
iff for any bad vertex $w$ adjacent to $v$, all the corresponding dull vertices
(i.e. those of degree $d_w-1$) are also either adjacent to $v$ or equal to $v$. 
It follows that, in the complement
graph $\bar{G}$, $v$ has the property that for any dull vertex $w$ not adjacent to $v$, all the corresponding
bad vertices (i.e. those of degree $d_w+1$) are not adjacent to $v$ either. It follows that $v$
satisfies the condition of lemma (\ref{dsc}) in $\bar{G}$. \qed  \end{proof} 

To prove the corollary, it suffices to observe that for every graph with a given degree 
sequence, the complement graph has the complement degree sequence.  \qed

This lemma serves to show that if we wish to discover whether
a graph is ds-reconstructible, 
then it is sufficient to determine the status of either one of $G$ or $\bar{G}$. 
The same applies to the task of determining whether a given degree sequence
is ds-reconstruction-forcing.

\begin{definition}
	A {\em stub} is half of an edge, such that each stub attaches to one vertex, 
	and each edge is the result of joining two stubs.
\end{definition}

We will also say that an edge or a stub is `bad' when it is attached to a bad vertex.

\begin{definition}
	A vertex set is {\em neighbourly} if no vertex outside the set has a degree one
	less than that of a vertex in the set. Thus, $K$ is neighbourly iff 
	$\{d_u : u \notin K\} \cap \{d_u-1 : u \in K\} = \emptyset$. A vertex set
	is {\em $d$-neighbourly} if it contains a vertex $v$ of degree $d$ such that no vertex outside the
      set has a degree one less than that of any vertex other than $v$ in the set.
      Thus, $K$ is $d$-neighbourly iff 
	$\{d_u : u \notin K\} \cap \{d_w-1 : w \in K - v\} = \emptyset$ and $d_v = d$.
       For example, a set of four vertices with degrees $[3,4,4,5]$ is 3-neighbourly if there is 
       no other vertex of degree 3 or 4, whether or not there is a vertex of degree 2.
\end{definition}

The significance of this concept is that if a vertex $v$ is adjacent to {\em all} the 
members of a neighbourly set, then those neighbours can be deemed `good' as far as 
$v$ is concerned. Indeed, it is obvious from the definitions that 
a vertex $v$ is ds-completable iff $N[v]$ is the union of one or more neighbourly
or $d_v$-neighbourly sets.

Any set that contains all the dull vertices is neighbourly. Any set $j$ consisting of
all the vertices of degrees in some range $d_{\rm min}(j) \cdots d_{\rm max}(j)$ is
neighbourly iff $d_{\rm min}(j)$ is good.

\begin{theorem}  \label{th.allv}
	A graph has all of its vertices ds-completable if and only if no pair of degrees differ by 1.	
\end{theorem}

\begin{proof}
({\em If}):	
If no pair of degrees differ by 1, then all vertices are good, so no vertex has a bad neighbour,
therefore every vertex is ds-completable. \qed\\
({\em Only if}): We will show that if there is a bad vertex, then there is a non-completable vertex.
Let $v$ be a bad vertex which has degree $d$. Let $\alpha$ be the set of vertices of degree $d-1$.
Consider the edges of $v$. Suppose $a$ of these edges go to vertices in $\alpha$. Then $(d-a)$ of them
go to vertices not in $\alpha$.	
In order that all vertices be completable, we require 
in particular that each of the $(d-a)$ neighbours of $v$ that are not themselves in $\alpha$ must be adjacent to all of $\alpha$. Therefore edges between this group and $\alpha$ will use up
\be
(d-a) |\alpha|   \label{alphastub}
\ee 
of the stubs on vertices in $\alpha$. Next consider the $a$ vertices in $\alpha$ that are adjacent to $v$.
In order that they should be completable, we require that each is adjacent to all the other
vertices in $\alpha$. This requires $a(|\alpha|-a)$ edges between these $a$ and the
others in $\alpha$, plus a further $a(a-1)/2$ edges among the $a$ vertices. Hence the number of
stubs on vertices in $\alpha$ used up by the required edges among the vertices in $\alpha$ is
\be
2 a(|\alpha| - a) + a(a-1) .
\ee
Adding these to the number given by (\ref{alphastub}), and including also the $a$ edges
between $\alpha$ and $v$, we find that the number of stubs on vertices in $\alpha$ has to be
greater than or equal to
\be
(d-a) |\alpha| + 2 a(|\alpha| - a) + a(a-1) + a = |\alpha|(d+a) - a^2.
\ee
However, the total number of stubs on vertices in $\alpha$ is $(d-1)|\alpha|$. In order
that this should equal or exceed the number we just found to be required, one needs
$|\alpha| \le a^2/(a+1)$, i.e. $|\alpha| < a$, but this is clearly impossible. Hence
we have a contradiction, and it follows that at least one of the vertices is not completable. \qed
\end{proof}

The graphs having all good vertices form a type of error-correcting code, in the following sense.
If one party is going to send graphs to another, and it has been agreed beforehand that only
graphs of a given degree sequence will be sent, then if all degrees are good the receiver can determine
perfectly, from the received graph, which graph was sent, even when the communication channel
has a noise process whose effect is, with some probability, to delete one vertex
and its edges. (It is not my
intention to suggest that this is a physically realistic scenario, only to point out the mathematical
fact.) More generally, the degree sequence would not need to be agreed in full between the
communicating parties. If would suffice,
for example, that throughout the whole set of degree sequences which
are allowed, no degree differs from another by one.

The rest of the paper is concerned with trying to discover, by examining the degree
sequence $\pi$, whether or not it so constrains the graph that a completable vertex must
be present. For each $\pi$ this is like solving a logic puzzle, somewhat reminiscent of
the popular {\em Sudoku} puzzle. The degree sequence provides a `rule'
on how many edges each vertex has, and we have to fit these edges into the graph without
breaking the rule `no vertex is allowed to be completable'. If such a puzzle has
no solution, then $\pi$ is ds-reconstruction-forcing. One can enjoy this game
as a logical exercise in its own right, without regard to Kelly and Ulam's reconstruction
conjecture. The fact that we can write down such preliminary general observations as
lemmas \ref{lemswtich} and \ref{lem.comp} and theorems \ref{dsc} and \ref{th.allv} suggests
that we have a reasonably well-defined area of mathematics to explore. 
However, it will emerge, in the subsequent sections, that the puzzle is a hard one!

\section{Easily recognised forcing sequences}

For each result to be presented, we shall give an example degree sequence
to which the result applies. For this purpose we adopt a notation
where we simply list the degrees, as in, for example, $[[1222]33\,56]$
(only single-digit examples will be needed), with square brackets
or bold font used to draw attention to vertices featuring in the argument.

\begin{lemma} \label{obv}
A degree sequence having one or more of the following properties is forcing:\\
\begin{enumerate}
\item There is no or only one bad vertex $($e.g. $[3333334])$.
\item There is no or only one dull vertex  $($e.g. $[2333555])$.
\item There is a single vertex which is both bad and dull, and there are 2 bad degrees
$($e.g. $[1111233])$.
\item The sum of the degrees of bad vertices is less than $|V| + ( |B| \mbox{ \rm mod } 2)$, 
where $B$ is the set of bad vertices.
$($e.g. $[1111[223]5])$.
\end{enumerate}
\end{lemma}

\begin{proof} 
Parts 1--3 are merely ways of expressing examples of
case 4 of lemma \ref{simple}. Part 4
is owing to the fact that if there are too few bad edges then there must be a vertex with no bad
neighbour. Let us define $n = |V|$ and $\delta_B \equiv |B| \mbox{  mod } 2$.
The least number of stubs on bad vertices 
required in order that every vertex can have at least one bad neighbour is
$n + \delta_B$ because we require at least $(n-|B|)$ edges
extending between $B$ and good vertices, and at least $\lceil |B|/2 \rceil$ edges
extending between bad vertices. The former use up one bad stub each, the latter use up two bad
stubs each. \qed  \end{proof}

\begin{lemma}  \label{BD}
If either of the following conditions hold:
\begin{enumerate}
\item there are exactly two bad degrees, and there is a dull vertex $v$ of unique degree, and a bad vertex $u$ of unique degree,
with $d_u \ne d_v + 1$, 
$($e.g. $[{\bf 2}3333{\bf 4}])$
\item there are exactly three bad degrees, and there are two vertices $u$, $v$, each both bad and dull and of unique degree,
$($e.g. $[22{\bf 34}555])$
\end{enumerate}
then either $\Gv$ is ds-completable or $G\! - \! u$ is ds-completable.
\end{lemma}

\begin{proof} 
Choose the labels $u,v$ such that $d_u \ne d_v + 1$.
If there is an edge $uv$ then $\Gv$ is completable.
If there is no edge $uv$ then $G-u$ is completable. (For further elaboration, see lemma \ref{unique}.)
 \qed  \end{proof}

\section{Application to graph generation}  \label{s.generate}

As added motivation for the pursuing the subject further, we present in
this section an application.

The concept of `completing' a card by adding a vertex to it arises in the
area of graph generation. For example, one may wish to generate a set of
graphs of given degree sequence. A typical procedure is to generate smaller graphs
and extend them, and then eliminate duplicates by isomorphism checking. If one
first observes that the target degree sequence is forcing, then one knows that, for
each graph one is seeking, there exists a subgraph which is uniquely completable.
Therefore, in the graph generation algorithm, one may bypass those subgraphs
which are not uniquely completable (a property which is easy to check), in confidence
that a uniquely completable subgraph will eventually turn up, and then a graph
can be generated from that subgraph. In this way one obtains a
reduction in the amount of generation of isomorphic duplicates,
for only a modest computational effort.

Here is an example. Suppose we wish to generate all graphs having the degree
sequence $[2333555556]$ (there are 4930 such graphs). First we note, from lemma \ref{BD}
part 1, that the sequence is forcing. After deleting the vertex of degree 2,
there are 5 possibilities for the degree sequence of a graph on 9 vertices
which can be used as a template to generate the graphs we want. These are
$[223555556]$, $[233455556]$, $[333445556]$, $[233555555]$, $[33455555]$.
If we were to start by creating all graphs realising these degree sequences
and completing them by adding a vertex
in all ways that give the target degree sequence, then we would
generate $(119 + 1068 + 1810 + 6 \times 96 + 5 \times 351) = 5328$ graphs.
Isomorphism checking would then be required to reduce this set. If instead
we use the ds-reconstruction-forcing property, then we proceed as follows.
Let $u$ be the vertex of degree 2, and $v$ be the vertex of degree 6.
First, we note that it is the cases where $u$ and $v$ are
adjacent that give a non-uniquely-completable card after $u$
is deleted. Equally, it is precisely those cases where deletion of $v$
will yield a completable card. Therefore the set of degree sequences
to be considered is those where $u$ is removed if $u$ and $v$ are non-neighbours,
and where $v$ is removed otherwise. This yields the sequences
$[223555556]$, $[233455556]$, $[333445556]$ as before, and also the sequences
$[122244555]$, $[122344455]$, $[123344445]$, $[133344444]$. Now we consider
all graphs of order 9 that realise one of these sequences. In every case,
each such graph will yield one and only one graph of order 10 having the target
degree sequence, and no two of these will be isomorphic because they each have
a different degree neighbourhood of one or both of $u$, $v$.
The total number of graphs generated is now $(119+1068+1810+88+684+963+198)
= 4930$ and no isomorphism checking is required on the set of graphs of order 
10 that has thus been generated.

This does not rule out that isomorphism
checking may be required to generate the graphs required on 9 vertices, but
we note that one of the order-9 sequences under
consideration is itself forcing.

\section{Recognising forcing sequences more generally}

Consider now the degree sequence [112223 5666]. 
In order to prove that this sequence
is forcing, we shall suppose that it is not, with a view to obtaining a contradiction. 
Label the vertices by their degree, and suppose no vertex is completable.
Gather the vertices of degree 5 or 6 into a set called $j$, and the other vertices
into a set called $i$. Each vertex of degree 6 must have at least 3 edges to vertices
in $i$, and each vertex of degree 5 must have at least 2 edges to vertices in $i$,
so there are at least 11 edges from $j$ to $i$. But this is equal to the number of stubs
in $i$, so we deduce that all edges from vertices in $i$ go to $j$. The 
pendants cannot be adjacent to 5, because 5 is good, so 5 is adjacent to at least one 2.
But the other edge from that 2 must go to a vertex in $j$, so it must go to a 6.
Therefore that 2 has a neighbourly neighbourhood, which contradicts the assumption
that no vertex is completable. It follows that there must be a completable vertex.

In order to formulate such arguments more generally, 
divide the vertices into distinct sets, such that 
all vertices of a given degree
are members of the same set, but it is allowed that a set may contain
vertices of more than one degree. The vertex sets will be labelled by indices
$i,j,k,c,g,h,p,q,s,t$. We shall choose these labels such that a set with label $i$ will 
typically have vertices of low degree and a set with label $j$ or $h$
will typically have vertices of high degree. Throughout the rest of the paper,
the following conventions will be observed when choosing labels for sets of vertices:
\begin{enumerate}
\item $V$ is the set of all vertices in the graph,
$B$ is the set of bad vertices, $D$ is the set of dull vertices. 
\item For any set $i$, define $\bar{i} \equiv V \setminus i$. Thus the set of good vertices is $\bar{B}$.
\item Sets labelled $i$ and $j$ are distinct.
\item $c = i \cap B,\; g = i \setminus c,\; p = \bar{i} \cap D,\; q = \bar{i} \setminus p.$
Thus $g,c,p,q$ are all distinct and together they account for all the vertices in $V$. Note that, by
definition, $g$ is good, $c$ is bad, $p$ is dull, and $q$ has no dull vertices.
\item The label $k$ is only applied to sets which are bad by definition;
the label $h$ is only applied to sets which by definition include the vertex or vertices
of highest degree.
\item $u,v,w$ refer to single vertices.
\item It will not be necessary to distinguish between a single vertex and the set whose
only member is that vertex.
\end{enumerate}

Let $\epsilon(j,k)$ be the number stubs in $j$ that are involved in edges between
$j$ and $k$. When $j$ and $k$ are distinct, this is equal to the number of edges between
$j$ and $k$. When $j=k$ it is equal to twice the number of edges within $j$, and more
generally $\epsilon(j,k) = \epsilon(j, k \setminus j) + \epsilon(j, k \cap j)$.
Define $\sigma(i,j)$ ($\tau(i,j)$) to be a lower (upper) bound on $\epsilon(i,j)$, i.e.
\be
\sigma(i,j) \le \epsilon(i,j) \le \tau(i,j) .
\ee
The bound $\sigma$ ($\tau$) will be raised (lowered) as we take into
account more and more information about the graph, based on
its degree sequence, under the assumption that it is a
graph and has no completable vertex. If eventually it is
found that $\tau(i,j) < \sigma(i,j)$ for some $i,j$ then
the conditions cannot be met: it amounts to a contradiction.
In this case we deduce that there must be a completable vertex
and therefore the degree sequence is forcing.

For any vertex-set $i$, define $n_i = |i|$, define
$\kappa(i,d)$ to be the number of members having degree $d$,
and $m_i$ the sum of the degrees:
\be
n_i &=& |i|  \\
\kappa(i,d) &=& \left| i \cap \left\{v : d_v = d \right\}  \right| \\
m_i &=& \sum_{v \in i} d_v
\ee
$m_i$ is the total number of stubs on vertices in $i$. It will also be useful to define
\be
\xi_i &=& n-1-n_i  , \\
\chi_{ij} &=& m_j - (n-1-n_i) n_j .
\ee
For any vertex not in $i$, $\xi_i$ is the number of other vertices also not in $i$.

The following conditions follow from the definition of a graph; 
they must hold for any graph whether or not there is
a completable card, for distinct $i,j,k$:
\be
\epsilon(i,\,i \cup \ibar) &=& m_i    \label{epsitot}  \\
\sigma(i,j) &\ge& 0  \label{sigpositive}  \\
\tau(i,j)     &\le& \min \{ n_i n_j,\; m_i,\; m_j \}   \label{maxstubs} \\
\epsilon(i,j) &\ge& \sigma(j,i) \\
\epsilon(i,j)     &\le& \tau(j,i) \\
\sigma(j,k) &\ge& \sigma(i \cup j, k) - \tau(i,k)      \label{sigmtau} \\
\tau(j,k)     &\le& \tau(i \cup j, k) - \sigma(i,k)   \label{taumsig} \\
\sigma(i,j)  &\ge & \chi_{ij}   \label{minedge}  \\
\sum_i \sigma(i,j) &\le& m_j \,\le\, \sum_i \tau(i,j)    \label{sumeps}
\ee
where in (\ref{sumeps}) each sum indicates a sum over distinct sets that together contain all
the vertices in the graph. Most of the above are obvious. To prove (\ref{minedge}),
note that since a vertex of degree $d$ is non-adjacent to $n-1-d$ vertices, it can only be
non-adjacent to at most this number from any set of other vertices. Therefore
a vertex of degree $d$ is adjacent to at least $n_i-(n-1-d)$ members of any vertex set of 
size $n_i$ (and this is true whether or not $n_i > n-1-d$). 
The result follows by summing this over all members of $j$:
\[
\sum_{v \in j} \left( n_i + d_v + 1 - n \right) = m_j - n_j \left( n - n_i - 1 \right).
\]

The following also hold for any set of vertices in any graph:
\be
\epsilon(i,i)  &\le& \tau(i,i) - (\tau(i,i) \mbox{ mod } 2) ,        \label{oddtau} \\
\epsilon(i,\bar{i})  &\ge& \sigma(i,\bar{i}) + ( (m_i - \sigma(i,\bar{i})) \mbox{ mod } 2) . 
\label{oddsig} 
\ee
These both follow from the fact that the number of stubs in a graph is always even. In
the first case, if we have established an upper bound $\tau(i,i)$ which is odd, then we know
the true number is at most one less than this. In the second case, 
$(m_i - \sigma(i,\bar{i}))$ is an upper bound on the number of stubs 
remaining in $i$ after all edges out of $i$ are accounted for. If this is odd
then the true value is at most one less than this, and therefore $\epsilon(i,\bar{i})$
is at least one more than $\sigma(i,\bar{i})$. These conditions are useful
for tightening constraints on $\epsilon$.

Define the following functions from vertex sets to binary values:
\be
\delta_i^j &:=& \left\{ \begin{array}{ll} 
	1 & \mbox{ if } i \cap j \ne \emptyset \\
	0 & \mbox{ if } i \cap j = \emptyset \end{array}  \right. \\
\nu_i &:=& \left\{ \begin{array}{ll} 
	1 & \mbox{ if $i$ is neighbourly}  \\
	0 & \mbox{ otherwise } \end{array}  \right.    \\
\mu_i &:=& \left\{ \begin{array}{ll} 
	1 & \mbox{ if $\ibar$ is neighbourly}  \\
	0 & \mbox{ otherwise } \end{array}  \right.  
\ee

\subsection{General conditions}

If a graph has no completable vertex, then 
\be
\epsilon(i,j) &\le& m_i - (1 - \delta_j^B) n_i ,           \label{jgood} \\
\epsilon(i,j) &\le& m_i - \nu_j \kappa(i, n_j).            \label{jneighbourly}  
\ee
{\em Proof.} If $\delta_j^B = 1$ then (\ref{jgood}) follows from (\ref{maxstubs}); if $j$ is good (i.e.
$\delta_j^B = 0$) then no vertex in $i$ can have all its edges to $j$, or that vertex would
be completable, hence $\epsilon(i,j)$ must leave at least one stub unused on each vertex in $i$,
and (\ref{jgood}) follows. If $j$ is neighbourly, then after edges from $j$ to $i$ are accounted for
there must be an unused stub for each vertex in $i$ of degree $n_j$, which gives (\ref{jneighbourly}). \qed

Define the notation $\left[ x \right]^+$, where $x$ is a number, to mean
\be
\left[ x \right]^+ \equiv \max(0, x).
\ee
If there is no completable vertex then
\be
\lefteqn{\epsilon(i,j) \,\ge\, \chi_{ij}}   \nonumber  \\
&& + \left(1-\delta_i^D\right) \left[n_j + \left[ n_j n_s - \tau(s,j) \right]^+\right]   \nonumber \\
&& + \delta_i^D \mu_i \left[ \kappa\left(j, \xi_i \right)
  + \left[ \kappa(j, n-n_i) - \tau(c,j) \right]^+ \right]  \nonumber \\
 \label{xminedge} 
\ee

where $s = V \setminus (i \cup j \cup D)$ and $c = i \cap B$.
{\em Proof.}
This is eqn (\ref{minedge}), modified so as to include conditions that no vertex is completable.
First suppose $\delta_i^D = 0$. One has $\epsilon(i,j) = \chi_{ij}$ only on the assumption that each
vertex in $j$ has a closed neighbourhood containing all of $\bar{i}$, but if $\delta_i^D = 0$ then
$D \subset \bar{i}$ so this will result in all vertices in $j$ being ds-completable. To avoid this,
we must remove one edge from each vertex in $j$ to a dull vertex, and let that edge go to $i$ instead.
This accounts for the $n_j$ term in the first square bracket. We then note that we are still assuming
that every vertex in $j$ has an edge to every vertex in $s$ (hence $n_j n_s$ edges), but if this is more
than are possible (i.e. larger than $\tau(s,j)$) then we must correct for this, hence the second term
in the first square bracket. 

Next suppose $\delta_i^D = 1$. 
In this case we consider the vertices 
in $j$ whose degree is such that their closed 
neighbourhood must be $\bar{i}$ if they have no edge to $i$. But if $\bar{i}$ is neighbourly,
then this will not happen if no vertex is ds-completable, so such vertices must have at least
one edge to $i$. This accounts for the first term in the second square bracket in (\ref{xminedge}). 
Next, consider the vertices in $j$ of degree $n-n_i$. The calculation using $\chi_{ij}$ allows that
each of these sends just one edge to $i$. But if $\bar{i}$ is neighbourly, then that edge must go
to a bad vertex in $i$, and if $\tau(c,j)$ is small (for example, because $m_c$ is small) 
then this will not be possible for all the $\kappa(j,n-n_i)$ cases. The second term in the second
square bracket in (\ref{xminedge}) corrects for this.
\qed 

The combination of 
eqns (\ref{jgood})--(\ref{jneighbourly}) with (\ref{xminedge})
is sufficient to prove the forcing behaviour of many 
degree sequences. For example, $[2222][4455]$ is forcing by
(\ref{jgood}), (\ref{xminedge}) and
$[1222]33[56]$ is forcing by (\ref{jneighbourly}), (\ref{xminedge}).
For $n=8$ there are 293 degree sequences having no isolated or full vertex,
and of these, 195 are forcing. 
Eqns (\ref{jgood})--(\ref{xminedge}), together with lemma \ref{obv}, are sufficient 
to prove the forcing condition in all but 29 of the 195 forcing cases.  

In order to apply the constraints to any given degree sequence, one must make a judicious 
choice of how the vertices are partitioned into sets. One wants to choose a case with
high $m_j$ and low $m_i$, and if $j$ is good or neighbourly, then so much the better, since
this will tighten the bound set by (\ref{jgood}) or (\ref{jneighbourly}) or both.
Using (\ref{jgood}) and
(\ref{xminedge}), we have $m_i + n_j (\xi_i - 1 + \delta_i^D) \ge m_j +
\delta_i^D \mu_i \kappa(j,\xi_i)$
for a non-forcing sequence. For a given set $j$, the quantity $(m_i +
\xi_i n_j)$ decreases by $n_j - d_v$ for each vertex of degree $d_v$
added to set $i$. Therefore,
to obtain a tight constraint, one should include in $i$ all vertices with
$d_v < n_j$, in the first instance, and also vertices of degree $d_v = n_j$ if this 
makes $\delta_i^D \mu_i \kappa(j,\xi_i)$ increase. Further useful tests
can be obtained by considering the set $i \setminus D$ and/or $j \setminus B$.

For distinct sets $i,j$, if there is no completable vertex then
\be
\epsilon(i,j) \ge \chi_{ij} + \nu_j\left[ \kappa(s,n_j) + \kappa(j,n_j-1)\right]
\label{epsnuj}
\ee
where $s = V \setminus (i \cup j)$.
{\em Proof}. The argument is similar to that for eqn (\ref{xminedge}). $\chi_{ij}$ gives 
the minimum number of edges extending from $j$ to $i$ on the assumption
that each vertex in $j$ is adjacent to all of $\bar{i}$ except itself, and
on this assumption, any vertex in $s$ of degree $n_j$ will have all its edges
to $j$. But if $j$ is neighbourly then this is not allowed, so for each such member of $s$ one of the vertices in $j$ is not a neighbour, which implies a further edge to $i$. This gives rise to the $\kappa(s,n_j)$ term in
(\ref{epsnuj}). The $\kappa(j,n_j-1)$ is obtained in the same way, applying
the argument to members of $j$. \qed

In order to detect the forcing condition more generally, one must
examine the degree sequence in more detail. The rest of the paper presents
a series of constraints that are sufficient to identify all forcing sequences for
graphs on up to 10 vertices. I have not been able to discover any single simply stated
condition sufficient to distinguish all the forcing from the non-forcing sequences.
Instead there are a range of cases to consider, but some broad
themes emerge. One theme is that to avoid being
forcing a sequence must not result in vertices that `greedily' consume all the bad stubs, so that 
none are left for some other vertex. Another theme is that it is worth paying special attention 
to the case where there is a vertex whose degree is not shared by any other vertex.

Each constraint is introduced by
an example degree sequence whose forcing nature can be proved by the
constraint under consideration, but not by most (or in some cases any) of the other constraints.

\subsection{Conditions relating to bad stubs}

[2 444555 [777]].
For any good vertex set $j$, if the sequence is not forcing then
\be
m_B \ge \sigma(B,j) + n - n_j + \delta_B  .       \label{mBbound}
\ee
{\em Proof}. The edges from vertices in $j$ use up at least $\sigma(B,j)$ of the bad stubs.
There are $n-n_j$ vertices not in $j$, and these require at least one bad stub each. 
Therefore they require $n-n_j + \delta_B$ bad stubs. \qed

In order to apply this constraint, a good choice is to adopt for $j$ the set of good
vertices whose degree satisfies
\be
d > n - 1 - n_B - (1 - \delta_B^D).
\ee
Such vertices are liable to have more than one bad edge.

[222333 [666]7].
If the degree sequence is not forcing then, for any good set $j$,
\be
\sigma(B,j) \! &\ge& \sigma(g \cup B, j) - (m_g - n_g)  , \label{epsBj}    \\
\!\! \sigma(B,B) \! &\ge& 2 \sigma(c,k) + \left[ n_c - \sigma(c,k) \right]^+,  \label{epsBB}  \\
m_B &\! \ge& \! \sigma(B,j) + \sigma(B, B) + n \!- n_j\! - n_B  ,       \label{mBbound2}
\ee
where $c = i \cap B$, $g = i \setminus c$, and $k = \bar{i} \cap B$
for some $i$ distinct from $j$.
{\em Proof}. (\ref{epsBj}) follows from (\ref{sigmtau}) and (\ref{jgood}).
(\ref{epsBB}) expresses the fact that each edge between $c$ and $k$ uses two bad stubs, and any
vertex in $c$ not adjacent to a vertex in $k$ also requires a bad edge. 
(\ref{mBbound2}) is obtained by the same argument as for (\ref{mBbound}).  \qed

In order to use 
(\ref{epsBj})--(\ref{mBbound2}) one may use (\ref{xminedge}) or (\ref{epsnuj}) or another
method to obtain $\sigma(c,k)$ and $\sigma(g \cup B,j)$.

[11 334444 66]. For any degree sequence such that $D \cap B = \emptyset$, partition the vertices
into distinct sets $g,D,B,j$ where $m_g$ is small and $m_j$ is large. Then, if no
vertex is completable,
\be
\tau(D,j) &\le& \sum_{w \in j} \min(d_w, n_D-1) \le
n_j(n_D - 1),   \nonumber \\
\tau(g,j) &\le& m_g - n_g,    \nonumber \\
\tau(g,D) &\le& m_g - n_g,   \nonumber \\
\sigma(B,j) &\ge& m_j - n_j(n_j-1) - \tau(D,j) - \tau(g,j)   \nonumber \\
\sigma(D,B) &\ge& m_D - n_D(n_D-2) - \tau(D,j) - \tau(D,g),   \nonumber  \\
m_B &\ge& n_B+ \delta_B + \sigma(D,B) + \sigma(B,j) + n_g
\ee
{\em Proof}. The first constraint is the requirement that no vertex be adjacent to all the
dull vertices. The next two are examples of (\ref{jgood}), then we use (\ref{sigmtau})
twice (and require that no vertex in $D$ is adjacent to all the others in $D$) 
and finally require that all vertices have at least one bad edge. \qed

\begin{lemma} \label{taukk_lem}
For any vertex set $i$ such that $\bar{i}$ is neighbourly, let 
$p = \bar{i} \cap D$. If no vertex is completable, then
\be
\tau(p,p) \le n_p(n_p - 2) + \tau(i,p).  \label{taukk}
\ee
\end{lemma}
\begin{proof}
Observe that since $\bar{i}$ is neighbourly and $p$ is all the dull vertices in $\bar{i}$,
a vertex in $p$ may only be adjacent to all other vertices in $p$ if
it has an edge to $i$. Hence there are at most $\tau(i,p)$ vertices in
$p$ adjacent to $n_p - 1$ fellow members of $p$, and the remaining $(n_p - \tau(i,p))$ 
are each adjacent to at most $n_p -2$ members of $p$. The
total for $\tau(p,p)$ is therefore at most $\tau(i,p)(n_p-1) +
(n_p - \tau(i,p))(n_p-2) = n_p(n_p-2) + \tau(i,p)$. \qed
\end{proof}

[12 5556666 8].
If there are 5 degrees, alternating between good and bad (i.e. $D \cap B = \bar{D} \cap \bar{B} = \emptyset$), then label the five sets of vertices $g,c,p,k,h$ (in increasing degree order) and let $i = g \cup c$. 
These labels have been chosen in such a way that $p = \bar{i} \cap D$, so (\ref{taukk}) applies,
and $g,p,h$ are all good. If there is no completable vertex,
\be
\sigma(c,h) &\ge& \chi_{ih} - (m_g - n_g),  \nonumber  \\
\tau(i,p) &\le& (m_i - n_i) - \sigma(i,h), \nonumber  \\
\tau(p,p) &\le& n_p(n_p-2) + \tau(i,p) - (\tau(p,p) \mbox{ mod } 2), \nonumber  \\
\sigma(p,k) &\ge& m_p - (\tau(i,p)+\tau(p,p)+n_p n_h) ,\nonumber  \\
\sigma(c,k) &\ge& \sigma(p,k) - n_k(n_p - 1), \nonumber  \\
m_c &\ge& \sigma(c,h) + \sigma(c,k) .
\ee
{\em Proof}. The first inequality follows from (\ref{sigmtau}), (\ref{minedge}) and
(\ref{jgood}), noting that $h$ is good. The second inequality
follows from (\ref{taumsig}) and (\ref{jgood}), noting that $p$ and $h$ are both good. 
The equation for $\tau(p,p)$ is (\ref{taukk}) adjusted using (\ref{oddtau}). Next, the constraint on 
$\sigma(p,k)$ follows from (\ref{epsitot}) and (\ref{sigmtau}). 
The constraint on $\sigma(c,k)$ follows from the fact that any vertex
in $k$ that is adjacent to all of $p$ must have an edge to $c$. Finally,
we employ (\ref{sumeps}) and (\ref{sigpositive}) to constrain $m_c$. \qed

If we have a good set $j$ and a distinct set $i$, such that $\bar{i}$ is neighbourly, then if
no vertex is completable,
\be
m_B \ge n +\delta_B + \left(n_j - \tau(i,j) + \chi_{ij} \right) n_k
\ee
where $k = \bar{i} \cap B$. {\em Proof.} $(\tau(i,j) - \chi_{ij})$ is the maximum number of
vertices in $j$ that are not adjacent to all of $\bar{i}$, therefore $n_j - (\tau(i,j)-\chi_{ij})$
is the minimum number of vertices in $j$ that are adjacent to all of $\bar{i}$. If $\bar{i}$
is neighbourly, these vertices require a bad edge to $i$ in addition to the $n_k$ bad edges they
each have to $\bar{i}$. The result follows. \qed

\subsection{Conditions relating to unique vertices}

\begin{definition}
A {\em unique degree} is one which appears only once in the degree sequence.
For a vertex $v$ of unique degree $d_v$, define
\be
\alpha_v &\equiv& \{ w : d_w = d_v -1 \},   \nonumber \\
\beta_v  &\equiv& \{ w : d_w = d_v + 1 \}.  \label{alphabeta}
\ee
\end{definition}

\begin{lemma}  \label{unique}
{\bf Part A}.
If there is a vertex $v$ of unique degree $d$,
then if $v$ is not completable, 
it must not be adjacent to all the dull vertices not in $\alpha_v$
and it must be adjacent to a bad vertex not in $\beta_v$.
Thus, if $s \equiv (D-v) \setminus \alpha_v$ and $t \equiv (B-v) \setminus \beta_v$ 
then $\epsilon(v,s) < n_s$ and $\epsilon(v,t) \ge 1$, where $\alpha_v$, $\beta_v$ are
defined in (\ref{alphabeta}).

{\bf Part B}.
If there is a vertex $v$ of unique degree $d$ and $\epsilon(i,v)=0$ for a neighbourly set $i$,
then if $v$ is not completable, 
$\epsilon(v,\, s \cap \ibar) < |s \cap \ibar|$ and $\epsilon(v,\, t \cap \ibar) \ge 1$,
where $s$ and $t$ are defined as in part A.
\end{lemma}
{\bf Corollary 1.} If a vertex $v$ has unique degree and $B \setminus \beta_v \subset i$, 
for some set $i$, then $\epsilon(v, \, i \cap B) \ge 1$.\\
{\bf Corollary 2.} If there is a $d$-neighbourly vertex set $j$
with $n_j = d+1$, and a distinct set $i$, then $\epsilon(i,j) \ge \chi_{ij} + 1$. \\
{\bf Corollary 3.} If a vertex $v$ has unique degree, then
for any distinct sets $i,j$ such that $s \subset j$, $v \notin i$,
and $d_{\rm min}(s) \ge n-n_i$, $\epsilon(i,j) \ge \chi_{ij} + 1$. \\
{\bf Corollary 4.} If there are exactly 3 bad degrees, and 
there are two dull vertices $u,v$ with unique degrees $d_u, d_v \ne d_u+1$,
such that $u$ is good and $v$ is bad, then $\epsilon(u,\, \beta_v)
= \epsilon(v,\, \beta_u) = 1$ and $\epsilon(u,v)=0$ (see figure \ref{f.uv}).

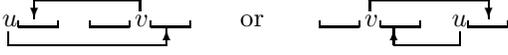
\begin{figure}
\begin{picture}(200,30)
\put(0,0){$u$}
\put(6,0){\put(0,0){\line(0,1){2}}\put(0,0){\line(1,0){15}}\put(15,0){\line(0,1){2}}}
\put(33,0){\put(0,0){\line(0,1){2}}\put(0,0){\line(1,0){15}}\put(15,0){\line(0,1){2}}}
\put(50,0){$v$}
\put(56,0){\put(0,0){\line(0,1){2}}\put(0,0){\line(1,0){15}}\put(15,0){\line(0,1){2}}}
\put(52,6){\put(0,0){\line(0,1){4}}\put(0,4){\line(-1,0){40}}\put(-40,4){\vector(0,-1){7}}}
\put(2,-2){\put(0,0){\line(0,-1){5}}\put(0,-5){\line(1,0){60}}\put(60,-5){\vector(0,1){6}}}
\put(90,0){\mbox{or}}
\put(120,0){
\put(0,0){\put(0,0){\line(0,1){2}}\put(0,0){\line(1,0){15}}\put(15,0){\line(0,1){2}}}
\put(17,0){$v$}
\put(23,0){\put(0,0){\line(0,1){2}}\put(0,0){\line(1,0){15}}\put(15,0){\line(0,1){2}}}
\put(50,0){$u$}
\put(56,0){\put(0,0){\line(0,1){2}}\put(0,0){\line(1,0){15}}\put(15,0){\line(0,1){2}}}
\put(19,6){\put(0,0){\line(0,1){4}}\put(0,4){\line(1,0){45}}\put(45,4){\vector(0,-1){7}}}
\put(53,-2){\put(0,0){\line(0,-1){5}}\put(0,-5){\line(-1,0){25}}\put(-25,-5){\vector(0,1){6}}}
}
\end{picture}
\caption{Illustrating corollary 4 to lemma \protect\ref{unique}. $u$ and $v$ are vertices
of unique degree; the horizontal brackets represent groups of vertices arranged in degree order; 
the arrows represent the presence of required edges in the associated graph if the sequence
is not forcing. There is no edge $uv$.}
\label{f.uv}
\end{figure}

\begin{proof}
The lemma follows immediately from the definitions;
the corollaries are simple examples. For the second and third corollaries, observe
that we cannot allow all the vertices in $j$ to have an edge
to the special vertex, because then that vertex would be completable. 
For the fourth corollary (c.f. figure \ref{f.uv}),
$v$ must be adjacent to a member of $\beta_u$ since these
are the only bad vertices with degree not equal to $d_v$ or 
$d_v+1$, and $v$ must not
be adjacent to $u$ because $u$ is the only dull vertex with degree not equal
to $d_v$ or $d_v-1$. It follows that $u$ must be adjacent to a vertex in $\beta_v$ since
these are the only bad vertices apart from $v$ whose degree is not $d_u+1$.
\qed
\end{proof}

Consider for example the sequences [1111{\bf 2}333 56] and [11111 3344{\bf 5}].
In the first case, let $v$ be the vertex of degree 2. Then the set $s$ consists
of one member, the vertex of degree 5.
This vertex is good, so it 
cannot be adjacent to any of the pendants (if they are not completable), therefore it
has to be adjacent to $v$, so we have $\epsilon(v,s)=1=n_s$
which breaks the condition stated in the lemma,
so the sequence is forcing. In the second case, the vertices of degree 3 cannot be adjacent to
the pendants, and nor can either of them have a closed neighbourhood containing all of $D$.
Hence they
must both be adjacent to the vertex $v$ of degree 5, so we have $\epsilon(v,s)=2=n_s$ so
this is forcing also. This argument is generalised in eqn (\ref{epssvms}).

Corollary 1 of lemma \ref{unique}, when used in conjunction with (\ref{xminedge}), 
suffices to prove that [1112 44{\bf 5}666] is forcing. The special vertex $v$ is the one 
of degree 5.  Let $j = \beta_v$, $i = \{ w : d_w < 3 \}$.
Since $v$ must be adjacent to the vertex of degree 2,
it uses up one of the stubs in $c = i \cap B$, with the result that
$\tau(c,j)=m_c-1=1$. Therefore (\ref{xminedge}) gives 
$\epsilon(i,j) \ge 3 + (3-1) = 5 = m_i$. But $\epsilon(i,j) \le m_i -1$ 
owing to the edge from $i$ to $v$, so we have a contradiction, hence the
sequence is forcing. 

Corollary 2 can be used to prove that [1112{\bf 3}455] is forcing, by taking
$j = \{ v : d_v \ge 3 \}$, $i =\bar{j}$.
One finds $\chi_{ij} = 5 = m_i$ so
$\epsilon(i,j) > m_i$ which contradicts (\ref{maxstubs}).

Corollary 3 can be used to prove that [1 333{\bf 4}55 778] is forcing, by taking
$d_v = 4$,
$j = \{ w : d_w > 6 \}$, $i = \{ w : d_w < 4 \}$.
We have $\chi_{ij} = 7$ so $\epsilon(i,j) \ge 8$ which contradicts (\ref{jneighbourly}).

[223334 {\bf 6}777]
If a good vertex $v$ has unique degree, then for any vertex set $i$,
\be
\epsilon(i,\, (\bar{i} \cap \bar{B}) \cup \beta_v + v) \le m_i - \sigma(i,v).       \label{corollary4}
\ee
{\em Proof}.
Any vertex in $i$ that is adjacent to $v$ cannot
have all its other edges to members of $\bar{B} \cup \beta_v$, since with $v$
they form a $d_v$-neighbourly set. \qed

[1233 {\bf 5}6666 8]. 
If there is a set $j$ containing a vertex $v$ of unique degree $d=n_j-1$, and such that
all bad vertices in $j$ have degree $n_j$, and
$\chi_{ij} \ge m_i - 2$ where $i = \bar{j}$, then the degree sequence is forcing.
{\em Proof}. Using corollary 2 to lemma \ref{unique}, $\epsilon(i,j) \ge m_i - 1$,
and therefore, using (\ref{oddsig}), $\epsilon(i,j) = m_i$. To avoid being
completable, the vertex $v$ must have a neighbour in $i$. But then that
neighbour will be completable, because all its neighbours are in $j$ and
the conditions are such that any set consisting of $v$ plus some members of $j$
is neighbourly. \qed

[1222 {\bf 4}55555].
If there is a unique good vertex $v$, and $d_v+1$ is the only bad degree greater
than $d_v$, then let $i = B \setminus \beta_v$ and 
$j = \{ d > d_v \}$. If $\sigma(i,j) \ge m_i - 2$ then the degree sequence is
forcing. {\em Proof}. If there is no completable vertex then $v$ must be
adjacent to a member of $i$. That member then requires an edge
to another member of $i$. Together these two edges require
3 stubs in $i$, but if $\sigma(i,j) \ge m_i - 2$ then there are
at most 2 stubs available after edges to $j$ have been accounted for. \qed

[11111 [33]44{\bf 5}]. 
If there is a non-dull vertex $v$ of unique degree,  let $s = D \setminus \alpha_v$
and $i = (V-v) \setminus D$.
We have $\tau(s,D) \le n_s(n_D-2)$ since if there is no completable vertex, then no vertex can
have a closed neighbourhood containing
all of $D$. We have that $V$ consists of distinct sets $i,\, D, \, \{ v \}$,
so, using (\ref{sumeps}),
\be
m_s \le \tau(s,i) + \tau(s,D) + \tau(s,v) 
\ee
which implies, using (\ref{jgood}) for $\tau(s,i)$, 
\be
\!\! \epsilon(s,v) \ge m_s - (m_i - (1 - \delta_s^B) n_i) - n_s(n_D-2) .  \label{epssvms}
\ee
If this number is $\ge n_s$ then, by lemma \ref{unique}, the degree sequence is forcing.

[1122 444 {\bf 6}77], [1122222 {\bf 5}67].
If there is a good dull vertex $v$ of unique degree $d \ge 3$, then let $j = \beta_v + v$
and $i =\{u : d_u < 3\}$. 
Either of the following conditions imply that the degree sequence is forcing:
\be
m_B < m_j + 1 \\
d  > n - (n_i + 3 + \mu_i)/2    \label{dnni}
\ee
{\em Proof.} In the first case, we have $m_B = m_k + (m_j-d)$ where $k = \bar{j} \cap B$. 
If the sequence is not forcing, then every
neighbour of $v$ requires an edge to $k$, and so does $v$. This is not possible
if $m_k < d+1$, which gives the condition. In the second case, we have that
$v$ is adjacent to at least $d - \xi_i$ members of $i$, and these neighbours may not be pendants,
therefore they have degree 2. No vertex of degree $d+1$ may be adjacent to any of
these, or we would have a vertex of degree 2 with a neighbourly neighbourhood. Hence
if $d+1 > n-1-(d - \xi_i)$ then the sequence is forcing. Furthermore, if $d+1=n-1-(d-\xi_i)$
then a vertex has a neighbourhood which contains all of $\bar{i}$ and all the pendants,
and therefore all of $D$ when $\bar{i}$ is neighbourly, so in this case $d$ must be
larger still to avoid forcing. (\ref{dnni}) follows.
\qed

[333[4444] {\bf 7}[88]].
If there is a good dull vertex $v$ of unique degree $d$, 
and no dull vertices of higher degree, then let
$h = \{ w : d_w > d \}$ and
$k = B \setminus h = B \setminus \beta_v$. If 
\be
m_k < \sigma(k, h) + 2 \sigma(k,v) + (\sigma(k,v) \mbox{ mod }2) + x
\ee
where $x = n-(n_k+n_h)-2 + \delta_k^D - \kappa(V,1)$ 
then the degree sequence is forcing. {\em Proof}.
By lemma \ref{unique},
the conditions are such that if the sequence is not forcing, then any vertex adjacent to $v$
must be adjacent to a vertex in $k$. The right hand side of the inequality gives a lower
bound on the number of stubs in $k$ required. Each edge contributing to $\sigma(k,h)$ requires
one stub in $k$. Each edge contributing to $\sigma(k,v)$ requires one stub in $k$, and then another
so that the vertex also has a neighbour in $k$. The $(\sigma(k,v) \mbox{ mod }2)$ term accounts
for the extra stub required when vertices in $k$ that neighbour on $v$ can't be merely joined in pairs.
Finally, $x$ is the number of good vertices
that must be adjacent to $v$ on the
assumption that we are making $\epsilon(k,v)$ as small as possible. 
For the number of vertices not in $k \cup h + v$ is $n-n_k-n_h - 1$, and 
$v$ must be adjacent to all but $\kappa(V,1) + (1-\delta_k^D)$ of them when $\epsilon(k,v)$ is minimised.
\qed

[123 5555 77{\bf 8}].
Let $d_v$, $d_u$ be the highest and next highest bad degrees, respectively. If $d_v$ is unique
and $d_v-1$ is good, then let $i = \{ w : d_w \le d_u \}$ and $h = \{ w : d_w \ge d_v - 1 \}$.
If $i \cup h \ne V$ and $\chi_{ih} \ge m_i - 2$ then the degree sequence is forcing. {\em Proof}.
Let $v$ be the vertex having degree $d_v$.
We have $D \setminus \alpha_v \subset i$, therefore, from lemma \ref{unique}, $v$ cannot be adjacent to all
of $i$. Since $B - v \subset i$, this implies that at least one vertex in $i$ must have an edge to
another vertex in $i$. This edge uses up two of the stubs in $i$, so under the condition 
$\chi_{ih} \ge m_i - 2$, every vertex in $j$ must be adjacent to all of $\bar{i}$ except itself.
It follows that each vertex not in $i$ or $j$ is adjacent to all of $\alpha_v$ and none of
$B - v$. Such vertices must be completable. \qed

If there is a bad unique vertex $v$, such that $\alpha_v$ is good, then let
$j = \alpha_v$. Then if there is no completable vertex,
\be
\tau(j,j) \le n_j(n_j-2) + m_B - d_v - 2.      \label{taujjv}
\ee
{\em Proof}. This is an example of relation (\ref{taukk}).
Let $c = B - v$. By lemma \ref{unique}, if there is no completable
vertex then $v$ must have an edge to $c$, and $v$ must not be adjacent to all the
dull vertices whose degree is one less than a member of $c$. 
Therefore there is an edge from $c$ to $v$, and there is an edge
between a vertex in $c$ and at least one other vertex not in $j$
(since every vertex requires a bad edge). Hence we have $\tau(c,j) \le m_c-2
= m_B - d_v - 2.$ One then obtains (\ref{taujjv})
by the same argument as for (\ref{taukk}). \qed

[3334 66666{\bf 7}].
In order to prove that this sequence is forcing, substitute
$\sigma(B,j) \ge m_j - (\tau(j,j) + m_g - n_g)$ into condition (\ref{mBbound}),
using (\ref{taujjv}) to provide a bound on $\tau(j,j)$, and
$g = \bar{B} \setminus \alpha_v$.

\subsection{Conditions relating to pendants}

If
\be
d_{\rm max}(\bar{B}) > n - 1 - \kappa(V,1) - x
\ee
where $x=1$ if $\kappa(V,2)=0$ and $x=0$ otherwise, then the sequence is forcing.
{\em Proof}. A good vertex having a degree above this bound is either adjacent to a pendant
(so the pendant is completable) or $x=1$ and it 
is adjacent to all the vertices except the pendants and itself; in the latter case it is
completable since $x=1$ implies there are no vertices of degree 2. \qed

[11111 3344{\bf 5}]. 
If $\kappa(V,1) = n/2$ and 
there is a vertex $v$ of unique degree $d_v = n - \kappa(V,1) = n/2$,
and $D = \{ w : d_v-2 \le d_w < d_v \}$, then the degree sequence is
forcing. {\em Proof}. This is an example of the situation treated by
eqn (\ref{epssvms}). In the case under consideration, 
$s = D \setminus \alpha_v$ is good and $i$ consists of pendants, so 
eqn (\ref{epssvms}) reads $\epsilon(s,v) \ge m_s - n_s(n_D-2)$. Also,
we have $m_s = n_s (d_v-2)$, so this becomes $\epsilon(s,v) \ge n_s(d_v-n_D)$.
Finally, $n_D < n - \kappa(V,1)$ since
pendants are never dull if there is no completable vertex, and by
assumption $v$ is also not dull. So if $d_v = n - \kappa(V,1)$ then
$d_v > n_D$ and we have $\epsilon(s,v) \ge n_s$,
which breaks the condition set by lemma \ref{unique}. \qed

[12 {\bf 45}66666 8].
If there are exactly 3 bad degrees, and these are $2, d+1, d+2$, with
$\kappa(V,2) = \kappa(V,d) = \kappa(V,d+1) = 1$, then if
\be 
d > (n - 2 - \kappa(V,1)) / 2   \label{d1}
\ee
then the sequence is forcing. {\em Proof.} 
Let $u$ be the vertex of degree $d$ and $v$ be the vertex of degree $d+1$. 
If there is no completable vertex, then by lemma  \ref{unique},
$v$ is adjacent to the vertex of degree 2, and not to the pendant. 
It follows that $u$ may not be adjacent to the vertices of degree
1 or 2, and therefore it may not be adjacent to $v$ either (or its closed neighbourhood
will be $d$-neighbourly). Hence all neighbours of $u$ are in $h = \{ w : d_w > d+1\}$.
Equally, since there is no edge $uv$, vertex $v$ also has $d$ neighbours in $j=h$. 
The number of vertices in $h$ adjacent to both $u$ and $v$ is therefore
greater than or equal to $2d - n_h$. At most one of these can have the vertex
of degree 2 as a bad neighbour, so if $2d - n_h > 1$ then there is a vertex
adjacent to $u$ and $v$ but not 2; such a vertex is completable. We have
$n_j = n-\kappa(V,1)-3$ and (\ref{d1}) follows. \qed

[1233$\ldots$ [good] $uv$].
If $\kappa(V,1) = \kappa(V,2) = \kappa(V,d) = \kappa(V,d+1) = 1$
for some $d$, and $\kappa(V,3) > 0$ and there are exactly three bad degrees
(namely $2,\, 3, \, d+1$), then if
\be
d \ge n-3
\ee
then the sequence is forcing. {\em Proof.} 
Let $u,v$ be the vertices of degree $d,d+1$. 
If there is no completable vertex, then the vertex of degree 2 must
be adjacent to $v$, and not to $u$. Therefore $v$ may not be adjacent
to the pendant (lemma \ref{unique}), and nor may $u$ (or the pendant
will be completable). Since the pendant requires a bad neighbour, it
is therefore adjacent to a vertex of degree 3. But if $d \ge n-3$ then
both $u$ and $v$ are adjacent to all the vertices of degree 3, so the
one with the pendant neighbour is completable.  \qed

\subsection{Conditions relating to $\sigma = \tau$}

The theme of the conditions in this section is that when edges from a set $j$ of high degree
use up all the available stubs in a set $i$ of low degree, then the vertices of intermediate
degree must have all their edges to each other or to $j$.

\begin{lemma}
	If $\sigma(i,\bar{i}) = m_i$ for a set $i$ with neighbourly $\bar{i}$,
	and there is no completable vertex, then
	\be
	\tau(i,p) \le \sum_{v \in i} \min(d_v, n_p-1)
	\ee
where $p = \bar{i} \cap D$. 
\end{lemma}
{\em Proof}. When $\sigma(i,\bar{i}) = m_i$, all edges from $i$ go to vertices
in $\bar{i}$. In this case, when $\bar{i}$ is neighbourly no vertex in $i$ may be adjacent to all of $p$, or it will be completable. \qed

This lemma suffices to prove that [[122233] 566 8] is forcing. In this example,
$p$ consists of the vertex of degree 5 and we find $\tau(i,p) = 0$. However
$\sigma(i,p) \ge \chi_{ip} = 2$, so we have $\sigma(i,p) > \tau(i,p)$ which
is a contradiction. \qed

\begin{lemma} \label{lemgtoh}
	If $\chi_{ij} \mu_i = m_i$ for distinct sets $i,j$ and there is no completable vertex, then 
	\be
	\epsilon(g,s) = m_g   \label{gtoh}
	\ee
	where $g = i \setminus B$ and $s = j \setminus \{ u : d_u = n-n_i \}$.
\end{lemma}
{\em Proof}. (c.f. figure \ref{f.gh}.) When $\chi_{ij} = m_i$, all the edges from $i$ go to $j$, and each
vertex in $j$ is adjacent to all the vertices in $\bar{i}$ except itself. 
The vertices of degree $n-n_i$ then have exactly one edge each to $i$, and 
when $\bar{i}$ is neighbourly, this edge must be to a bad vertex in $i$. Hence those
vertices have no edges to $g$. The result follows. \qed

\begin{figure}
	\begin{picture}(200,30)
	\put(0,0){\put(1,1){2233}\put(0,0){\line(0,1){2}}\put(0,0){\line(1,0){20}}\put(20,0){\line(0,1){2}}}
	\put(25,0){\put(1,1){666778}\put(0,0){\line(0,1){2}}\put(0,0){\line(1,0){29}}\put(29,0){\line(0,1){2}}}
	\put(0,9){\put(0,0){\line(0,-1){2}}\put(0,0){\line(1,0){10}}\put(10,0){\line(0,-1){2}}}
	\put(39,9){\put(0,0){\line(0,-1){2}}\put(0,0){\line(1,0){15}}\put(15,0){\line(0,-1){2}}}
	
	\put(5,10){\put(0,0){\line(0,1){4}}\put(0,4){\line(1,0){40}}\put(40,4){\vector(0,-1){6}}}
	\put(8,-2){\put(0,0){\line(0,-1){5}}\put(0,-5){\line(1,0){30}}\put(30,-5){\vector(0,1){6}}}
	\end{picture}
	\caption{An illustration of lemma \protect\ref{lemgtoh}.  
		The arrows represent cases where all the edges from 
		the vertex set at the start of the arrow must go to the vertex set at the end of the arrow,
		if no vertex is completable. The lower arrow represents the condition $\chi_{ij}=m_i$;
		the upper arrow represents eqn (\protect\ref{gtoh}).}
	\label{f.gh}
\end{figure}
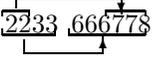

\begin{lemma}
For any vertex set $i$ whose maximum degree is less than
$\xi_i$, let $h = \{ u : d_u > \xi_i \}$. If 
there is no completable vertex and
\be
\left(\chi_{ih} +  \kappa(\bar{i}, \xi_i) \right) \mu_i  = m_i  \label{ihcond}  
\ee
then
\be
\epsilon(c, \, h) &\ge& 
n_h + \sum_{w \in h} \left[ d_w - n+n_c+1-\delta_c^D\right]^+,  \label{epsiBh} \\
\epsilon(i, \, s) &=& n_s ,  \label{epsink}
\ee
where $c = i \cap B,\; s = \{ u : d_u = \xi_i \}$.
\end{lemma}
{\em Proof}. If the condition holds then there is no room in $i$ for extra edges from $h$
in addition to the minimum number that go there on the assumption that every vertex in $h$
is adjacent to the whole of $\bar{i}$ except itself. When $\bar{i}$ is neighbourly, every 
vertex in $h$ then requires a bad neighbour in $i$ to avoid being completable, which
gives the first term in (\ref{epsiBh}). The second term allows for the further edges
which are unavoidable when vertices in $h$ have high degree. 
The condition also requires that each vertex in $s$ has one edge
to $i$, which gives (\ref{epsink}). 
\qed

[122 44555], [1222 556667].
Suppose (\ref{ihcond}) holds, and therefore so do (\ref{epsiBh}) and (\ref{epsink}).
Using (\ref{epsink}), (\ref{sigmtau}) and (\ref{jgood}) we have
\be
\epsilon(c,s) \ge n_s - m_g + (1 - \delta_s^B)n_g   \label{epsckg}
\ee
where $g = i \setminus c$. Hence if
\be
m_c < \sigma(c,h) + n_s - m_g + (1 - \delta_s^B)n_g  \label{mchkg}
\ee
where $\sigma(c,h)$ is given by any applicable methods, such as (\ref{xminedge})
and (\ref{epsiBh}), then the degree sequence is forcing.

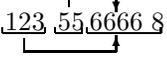
\begin{figure}
\begin{picture}(200,30)
\put(0,0){\put(1,1){123}\put(0,0){\line(0,1){2}}\put(0,0){\line(1,0){16}}\put(16,0){\line(0,1){2}}}
\put(20,0){\put(1,1){55}\put(0,0){\line(0,1){2}}\put(0,0){\line(1,0){10}}\put(10,0){\line(0,1){2}}}
\put(32,0){\put(1,1){6666 8}\put(0,0){\line(0,1){2}}\put(0,0){\line(1,0){28}}\put(28,0){\line(0,1){2}}}

\put(25,10){\put(0,0){\line(0,1){4}}\put(0,4){\line(1,0){18}}\put(18,4){\vector(0,-1){6}}}
\put(8,-2){\put(0,0){\line(0,-1){5}}\put(0,-5){\line(1,0){35}}\put(35,-5){\vector(0,1){6}}}
\end{picture}
\caption{An illustrative example. The arrows represent cases where all the edges from 
the vertex set at the start of the arrow must go to the vertex set at the end of the arrow,
if no vertex is to be completable. The lower arrow is an example of (\protect\ref{ihcond});
the upper arrow is an example of (\protect\ref{taukk}) combined with (\protect\ref{taumsig})
and (\protect\ref{sumeps}).}
\label{f.hk}
\end{figure}

[123 [55][6666 8]]. 
For any vertex set $i$ such that $\bar{i}$ is neighbourly, let 
$p = \bar{i} \cap D$ and $q = \bar{i} \setminus p$. 
If the condition (\ref{ihcond}) holds, and all the vertices
in $p$ have degree $\xi_i - 1$, 
and the only bad degree in $q$ is $\xi_i$,
then the degree sequence is forcing. {\em Proof}. 
If we assume there is no completable vertex, then using
(\ref{taukk}), (\ref{sigmtau}) and $m_p = n_p( \xi_i - 1)$
we obtain $\sigma(p,q) = n_p n_q$. Hence all
of the vertices in $q$ are adjacent to all of $\bar{i} \cap D$, 
(see figure \ref{f.hk} for an example). It follows that
when $\bar{i}$ is neighbourly, each vertex in $q$ requires an edge to a bad vertex in $i$. 
Now, when (\ref{ihcond}) holds, the vertices of degree $\xi_i$ each have
one and only one edge to $i$, and furthermore no vertex in $i$ is
only adjacent to good vertices in $q$. It follows that 
all the good vertices in $i$ are adjacent to vertices of degree 
$\xi_i$ in $q$, and such vertices in $q$ have no other edge to $i$.
Consequently they are completable. \qed

[2233 [666[778]]]. (See figure \ref{f.gh})
If $\chi_{ij} \mu_i = m_i$ for distinct sets $i,j$
with $d_{\rm min}(j) \ge \xi_i$ and
$(i \cap B) \cap D = \emptyset$, then if
\be
m_g > n_s(n_g - 1)   \label{mghg}
\ee
where $g = i \setminus B$ and $s = j \setminus \{ u : d_u = n-n_i \}$
then the degree sequence is forcing.
{\em Proof}. 
Using (\ref{gtoh}) we have $\epsilon(g,s) = m_g$ if there is no completable vertex.
When $\chi_{ij} = m_i$ we also have that every vertex in $s$ is adjacent to all of
$\bar{i}$, from which it follows that no vertex in $s$ may be adjacent to all of
$i \cap D$. But the condition $(i \cap B) \cap D = \emptyset$ implies
$i \cap D \subset g$. So we require that no vertex in $s$ is adjacent to
all of $g$, and therefore $\epsilon(g,s) \le n_s(n_g - 1)$. This
condition cannot be realized when $\epsilon(g,s) = m_g$ and 
(\ref{mghg}) holds. \qed

[112 [455][6666]].
For any vertex set $i$ such that $\bar{i}$ is neighbourly, let 
$c = i \cap B$, $p = \bar{i} \cap D$, $q = \bar{i} \setminus p$. 
If $n_q \ge m_c$
and $\sigma(i,q) = m_i$ and there is no completable vertex, then
\be
n_p(n_p+n_q-2) \ge m_p +n_q - m_c .   \label{nknk}
\ee
{\em Proof}. Any vertex in $q$
that is adjacent to all the vertices in $p$ requires an edge to $c$, because otherwise
its neighbourhood would be neighbourly. Hence at most $m_c$ of the vertices
in $q$ can be adjacent to all of $p$, therefore $\tau(p,q) \le n_p m_c + (n_p-1)(n_q - m_c)$.
Using (\ref{sigmtau}) we have $\sigma(p,p) \ge m_p - (\tau(p,q) + \tau(p,i)) = m_p - \tau(p,q)$
since $\epsilon(i,p)=0$ when $\sigma(i,q)=m_i$. Finally, since $\epsilon(i,p)=0$ no
vertex in $p$ may be adjacent to all the other vertices in $p$, because when $\bar{i}$
is neighbourly, this set includes
all the vertices whose degree is one less than a vertex in $\bar{i}$. Therefore
$\tau(p,p) \le n_p(n_p-2)$. Eqn (\ref{nknk}) follows by asserting $\sigma(p,p) \le \tau(p,p)$. \qed

[11122[3] [55][66]].
If there is a good set $j$ and a distinct set $i$ such that
$\sigma(i,j) = m_i - n_i$ and $\chi_{ik} > n_i - \kappa(i,\, n_j+1)$, where
$k = \{w : d_w = d_{\rm max}(j)+1\}$,
then the degree sequence is forcing. {\em Proof}.
Let $s = \{ w : d_w = n_j + 1\}$ and suppose the sequence is not forcing.
The condition $\sigma(i,j)=m_i - n_i$ implies that
after edges from $j$ have been accounted for, there is one stub left
on each vertex in $i$.
Each vertex in $s$ is therefore adjacent to
all of $j$, and has one edge to a vertex not in $j$. That final neighbour
cannot be in $k$ or the vertex is $s$ is completable, so $\tau(s,k) = 0$.
Also, there is at most one edge from $i$ to $k$ for each vertex in $i$,
so $\tau(i \setminus s, k) \le n_i - n_s.$ Under these conditions, 
eqn (\ref{sigmtau}) gives $\epsilon(s,k) \ge \chi_{ik} - (n_i - n_s)$.
The condition $\epsilon(s,k) \le \tau(s,k)$ then requires
$\chi_{ik} \le n_i - n_s = n_i - \kappa(i,\, n_j+1)$. \qed

[[12] 444455 [7]8].
If there is a good set $j$ and a distinct set $i$ such that
$\sigma(i,j) = m_i - n_i$, $\kappa(i,n_j) > 0$ and
$\sigma(s,k) > 0$, then the degree
sequence is forcing, where $s = \{ v : d_v = n_j \}$ and
$k = \{ v : d_v = d_{\rm max}(j) + 1 \}$. {\em Proof}.
After edges from $j$ are accounted for, each vertex in $i$ has
one stub remaining. We have $s \subset i$. If there is no
completable vertex, then the final edge from a vertex in $s$ cannot 
go to $k$, but the condition $\sigma(s,k) > 0$ requires that it
does, which is a contradiction. \qed

\subsection{Conditions involving $\sigma=\tau$ and unique vertices}

[1233 [{\bf 5}6666 8]]. Suppose there is a good vertex $v$ of unique degree 
$d = \xi_i$, and
$\sigma(i,j) = m_i$ where $j = \{w : d_w \ge d \}$, and $i$ is some vertex set
distinct from $j$. If $\bar{i}$ is neighbourly
and the only bad degree above $d$ is $d+1$, then
the sequence is forcing. {\em Proof}. The conditions are such that 
(\ref{corollary4}) gives $\epsilon(i,j) \le m_i - 1$ which contradicts
$\sigma(i,j) = m_i$. \qed

[1122 4{\bf 5}[6667]]. We now extend eqn (\ref{mchkg}) in the case where 
$n_s=1$, that is, there is a single vertex $v$ of degree $\xi_i$. If the condition
(\ref{ihcond}) holds, then there are no vertices of degree below $n_h$ in
$\bar{i}$, but there may be vertices of degree equal to $n_h$. Any such vertices
have all their edges to $h$ when the conditions holds, therefore they are not
adjacent to $v$. It follows that $v$, which must be adjacent to all 
of $\bar{i}$ except two vertices (itself and one other), 
has edges to all of $h$. Since $v$ is of unique degree $d$, it forms
a $d$-neighbourly set with $h$, so its edge to $i$ must go to a bad
vertex, and we have $\epsilon(c,s) = 1$ (whereas in this case
(\ref{epsckg}) merely gives $\epsilon(c,s) \ge 1 - m_g$ if $s$ is bad).
Hence under the assumed conditions, (\ref{mchkg}) can be replaced by
\be
m_c < \sigma(c,h)  + 1.
\ee

[112 4445{\bf 6}[78]]. This case is similar to the previous one.
If condition (\ref{ihcond}) holds, and 
$n_s = \kappa(\bar{i},\xi_i-1) = 1$,
$\epsilon(c,h) \ge m_c$,
and there are no bad degrees in $\bar{i}$ smaller than $\xi_i-1$, then the
degree sequence is forcing. {\em Proof}. 
Let $u$ be the vertex of degree $\xi_i-1$ and $v$ be the vertex of degree $\xi_i$.
When condition (\ref{ihcond}) holds, both $u$ and $v$ are adjacent to all of $h$.
$h+v$ is $d_v$-neighbourly, so
to avoid being completable, $v$ requires an edge to a bad vertex not in
$h$. When $\epsilon(c,h) \ge m_c$ this edge cannot be to $c$ so it must be to $u$.
But this implies that $u$ is adjacent to all of $(p-u) \setminus \alpha_u$
and therefore, since $u$ can have no edges to $i$, it must be completable.
\qed

[122 {\bf 4}555 [77]8]. 
If $\chi_{ij} = m_i - n_i$ for a good vertex set $j$ and there
is a good unique vertex $v \in (\ibar \setminus j)$ 
such that $(D-v) \setminus i \subset j$
then the degree sequence is forcing. {\em Proof}. Under the conditions,
$v$ can have no edges to $i$, and each member of $j$ is adjacent to
all of $\ibar \setminus j$. It follows that $v$ does not satisfy
the condition stated in part B of lemma \ref{unique}. \qed

[{\bf 1}222 [44]555 [8]],
[{\bf 1}222 [444]56 [8]]. 
Suppose $d_j = n-2-n_i \ne d_{\rm min}$ is a good degree, where 
$i = \{ w : d_w \le d_{\rm min}+1\}$ and $d_{\rm min}$ is the lowest
value in the degree sequence. Consider the case
where only one vertex $u$ has degree $d_{\rm min}$, 
and either there are 2 bad degrees $[d_{\rm min}+1, \, d_j+1]$
or there are 3 bad degrees $[d_{\rm min}+1,\, d_j+1,\, d_j+2]$, 
and in the second case
$d_j+1$ and $d_j+2$ are both unique.
Let $s = \{ w : d_w > d_j\} \setminus B$.
If $\sigma(i,s) = m_i - n_i$ then the sequence is forcing. {\em Proof}. 
Under the condition on $\sigma(i,s)$ 
there can be no edges from $i$ to good vertices not in $s$. 
Let $j = \{ w : d_w = d_j\}$. We have that $j$ is good, so 
$\sigma(i,j) = 0$. Therefore all the edges from vertices in $j$ go to $\bar{i}$. 
First suppose there are only 2 bad degrees. In this case
$j$ are the only dull vertices in $\bar{i}$, hence no vertex in $j$ may
be adjacent to all the other vertices in $j$. Hence, when the condition
$d_j = n-2-n_i$ holds, every vertex in $j$ is adjacent to every vertex in
$\bar{i} \setminus j$.
Now, by lemma \ref{unique}, $u$ must be adjacent to a bad vertex in $\bar{i}$.
That vertex is completable, since it is adjacent to all of $D$. 
Next suppose we have 3 bad degrees, the highest two of which are
unique. Let $v$ be the vertex of degree $d+1$ and $w$ be
the vertex of degree $d+2$.
Any vertex in $j$ must be adjacent to $w$ because otherwise the
vertex in $j$ would have a closed neighbourhood
containing all of $\bar{i} \cap D$ and consequently be completable.
Also, noting that $v$ is both bad and dull, by using corollary 4 of 
lemma \ref{unique} we have $\epsilon(u,w) = 1$. 
It follows that $D \subset N(w)$ so $w$ is completable. \qed

[1111122 [4]{\bf 5}6].
If $\chi_{ij} = m_i - n_i$ for good $j$ with
$d_{\rm min}(c) > n_j$ and there is a vertex $v$ of unique degree
$d = d_{\rm max}(j) + 1$ and the only bad degrees in $\ibar$ are
$d,\, d+1$, then the sequence is forcing. {\em Proof}. The condition
$\chi_{ij} = m_i-n_i$ implies that all but one stub on each vertex in $i$
is used by edges to $j$. The condition $d_{\rm min}(c) > n_j$ implies
that, in order to fill all but one stub of any vertex in $c$, edges from
all the vertices of $j$ are required, so every vertex in $c$ is adjacent
to all of $j$. By lemma \ref{unique}, vertex $v$ requires
a neighbour in
$(B-v) \setminus \beta_v = c$. That neighbour is completable. \qed

[111 {\bf 4}5555{\bf 6}7].
If there are two unique vertices $u,v$ satisfying the conditions
of corollary 4 of lemma \ref{unique}, with $d_u + 2 = d_v = \xi_i$,
and $\chi_{ij} + n_j = m_i$ for a non-dull set
$i$ and the distinct set $j = \{ w : d_w \ge \xi_i\}$,
then the degree sequence is forcing. {\em Proof}.
The conditions are such that $D \subset \bar{i}$ and
the only bad degrees are $d_v,\,d_v \pm 1$. When $D \subset \bar{i}$,
to avoid being completable, no vertex
may be adjacent to all of $\bar{i} - w$ where $w$ is one of the vertices
having the highest bad degree (since such vertices cannot be dull).
The condition $\chi_{ij} + n_j = m_i$
implies that each vertex in $j$ is adjacent to
all but two of the vertices in $\bar{i}$ (itself and another),
and the same is true of each vertex in $s = \{ w : d_w = \xi_i - 1 \}$.
It follows that, if there is no completable vertex, then $v$ and
every vertex in $s$ is adjacent to every vertex in $\beta_v$. 
Corollary 4 of lemma \ref{unique} also gives that $u$
is adjacent to one of the vertices in $\beta_v$. It follows
that that vertex must be completable, since its neighbourhood
includes all of $D$. \qed

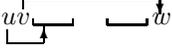
\begin{figure}
	\begin{picture}(200,30)
	\put(0,0){$uv$}
	\put(12,0){\put(0,0){\line(0,1){2}}\put(0,0){\line(1,0){15}}\put(15,0){\line(0,1){2}}}
	\put(40,0){\put(0,0){\line(0,1){2}}\put(0,0){\line(1,0){15}}\put(15,0){\line(0,1){2}}}
	\put(57,0){$w$}
	\put(8,6){\put(0,0){\line(0,1){4}}\put(0,4){\line(1,0){52}}\put(52,4){\vector(0,-1){6}}}
	\put(2,-2){\put(0,0){\line(0,-1){5}}\put(0,-5){\line(1,0){14}}\put(14,-5){\vector(0,1){6}}}
	\end{picture}
\caption{Illustrating lemma \protect\ref{unique3}. $u,v,w$ are vertices
		of unique degree; 
		the arrows represent the presence of required edges in the associated graph if the sequence is not forcing. There are no edges $uv,\,uw$. The illustration
		shows the case $\beta_v \ne \alpha_w$; the lemma applies equally when
		$\beta_v = \alpha_w$, and when there are further groups of good vertices.
	}
	\label{f.u3}
\end{figure}

\begin{lemma} \label{unique3}
If there are exactly 3 bad degrees, and there are three vertices $u,v,w$
of unique degree such that $d_v = d_u+1$, $u$ is good, $v$ is dull, 
$w$ is bad and not dull, then if none of $u,v,w$ are completable then
$\epsilon(u,v) = \epsilon(u,w) = 0$,
$\epsilon(v,w)=1$, $\epsilon(u, \beta_v) \ge 1$.
(see figure \ref{f.u3}).
\end{lemma}
\begin{proof}
	The edge $vw$ follows from lemma \ref{unique}. The same lemma then implies that
	$w$ is not adjacent to $u$ because the conditions are such that $\{u,v \} = (D-w) \setminus \alpha_w$. It then follows that $u$ must be adjacent to $\beta_v$ and not to $v$. \qed
\end{proof}

[222 4{\bf 5}[66667]].
When the conditions of lemma \ref{unique3} apply, with $d_w > d_v$,
$u,v,w \in \ibar$, $d_v = \xi_i-1$, $j = \{ x : d_x \ge \xi_i \}$
and $\chi_{ij} + n_j = m_i$, the degree sequence is forcing. {\em Proof}.
The condition $\chi_{ij} + n_j = m_i$ implies that $v$ has no edge to $i$. It
follows that all its edges go to $\bar{i}$, and since $d_v = \xi_i-1$ it must
be adjacent to all of $\bar{i}$ except itself and one other vertex. By lemma 10,
there is no edge $uv$, so we deduce $N(v) = j$. The conditions are such that
$j+v$ is $d_v$-neighbourly, so then $v$ is completable. \qed

[{\bf 1}23333 [666]{\bf 7}].
If the conditions of lemma \ref{unique3} apply, such that $j = \alpha_w$
is good, then let $i$ be a distinct vertex set such that $(B - w) + u \subset i$.
If $\sigma(i,j) = m_i - n_i$ then the degree sequence is forcing. {\em Proof.}
After edges from $j$ have been accounted for, each vertex in $i$ has one stub
remaining, and that stub must be used to provide an edge to a bad vertex. 
Hence this final edge cannot be to $u$, for each vertex in $i$. Hence the only
way $u$ can have a bad edge is if it is adjacent to $w$, but the edge $uw$
is ruled out by lemma \ref{unique3}. \qed

\subsection{Conditions involving $\sigma=\tau$, unique vertices and bad stubs}

[223 {\bf 5}[667777]]. Suppose we have distinct sets $i,\, j = \{ w : d_w \ge \xi_i \}$ 
such that $\bar{i}$ is neighbourly and (\ref{xminedge}) gives $\sigma(i,j) = m_i$. 
If $\kappa(V, \, \xi_i - 1) = 1$, $\kappa(j, \xi_i) = 2$,
$\kappa(j, \, \xi_i+1) \ge m_c$ where $c = i \cap B$ and there are no dull degrees
in $\bar{i}$ except $\xi_i$ and $\xi_i-1$, 
then the degree sequence is forcing. {\em Proof}. Let $v$ be the vertex of degree $\xi_i-1$, 
let $s = j \cap D$, and suppose no vertex is completable. 
$v$ is not in $j$ so it can have no edges to $i$. Therefore $v$ 
cannot be adjacent to both the vertices of $s$, or it will be completable. 
It follows that the vertices of $s$ are adjacent, since whichever one is
not adjacent to $v$ must be adjacent to all other members of $j$ when (\ref{xminedge})
gives $\sigma(i,j) = m_i$. The degree of $v$ is such that it must be adjacent to at least
one of $s$, so we deduce that a vertex in $s$ has all of $\bar{i} \cap D$ in its closed
neighbourhood. 
Such a vertex requires a bad edge to $i$ to avoid being completable,
but under the given conditions, all the bad stubs in $i$ have been used up by edges to
vertices of degree $\xi_i + 1$. It follows that there must be a completable vertex. \qed

[122 555{\bf 6}778], [112 4445{\bf 6}[78]].
Suppose there is a set $i$ such that condition (\ref{ihcond}) holds
and $\kappa(V,\xi_i)=1$ and $\epsilon(c,h) \ge m_c$.
Then if either there is no bad degree less than $\xi_i$ in $\ibar$,
or there is one and it is unique and equal to $\xi_i-1$, 
then the degree sequence is forcing.
{\em Proof}. Let $u,v$ be the vertices of degree $\xi_i-1, \,\xi_i$.
When the conditions apply, all the vertices of degree above $d_v$ are
adjacent to $v$, and $v$ has no edge to $c$. It follows that, to avoid
being completable, $v$ requires
an edge to a bad vertex not in $(h + v) \cup c$, 
because $h+v$ is $d_v$-neighbourly. If there is no such vertex,
then this cannot happen, so the sequence is forcing. If 
there is a single such vertex ($u$) and its degree is $\xi_i-1$, then 
$u$ has edges to all of $v$ and $h$ and none to $c$, therefore
its closed neighbourhood is $d_u$-neighbourly, so it is completable.
\qed

[22333 555{\bf 6} 8]. Let $j$ be the set of good vertices of degree above some value,
and $g$ be the set of all the other good vertices. Then, if there is no completable vertex,
and using (\ref{sigmtau}), (\ref{jgood}),
\be
\epsilon(B,j) \ge m_j - n_j(n_j-1) - (m_g - n_g) .  \label{jgeps}
\ee
Now suppose we have that, when we use the right hand side of
 this expression to give $\sigma(B,j)$ in (\ref{mBbound}),
we find that the bound (\ref{mBbound}) is saturated. 
In this case, any argument which leads to the conclusion that $\epsilon(B,j)$ is in fact
larger than and not equal to $m_j - n_j(n_j-1) - (m_g - n_g)$
will lead to the condition (\ref{mBbound}) not being satisfied, so the degree sequence is forcing. 
We consider the case where there are exactly two bad degrees, $d_k$ and $d_v$,
and only one vertex $v$ has the higher of these
two degrees, and $\alpha_v \subset j$.
Let $k = B - v$. If no vertex is completable,
then, on the assumption that (\ref{jgeps}) is an equality, we must have
\be
\epsilon(k, \bar{B}) \ge 1 + \sum_{w \in j} \max\left(1, \, d_w +1+n_B-n\right).  \label{epscgsum}
\ee
For, in the conditions under consideration, each vertex in $j$ is adjacent to all the others in $j$,
and therefore requires an edge to $k$ in order that its closed neighbourhood is not neighbourly. 
Each one of 
those of degree higher than $n-n_k-2 = n-1-n_B$ requires $d_w-(n-1-n_B)$ edges to $k$ since 
the closed neighbourhood may not contain all of $D$, and the conditions are such that $k$ is not dull.
This gives the sum in (\ref{epscgsum}). The additional $1$ is owing to the fact
that $v$ cannot be adjacent to all the vertices of degree $d_k-1$, 
therefore at least one of these
requires an edge to $k$. If we now have $\epsilon(k,\bar{B}) > m_k - n_k$ then the condition
(\ref{jgood}) is not satisfied, so the degree sequence must be forcing.

[22{\bf 3}4 [66666]{\bf 7}].
If there are exactly three bad degrees, and there is a dull bad vertex $u$
of unique degree, and a bad vertex $v$ of unique degree $d_v \ge d_u + 4$,
and only one good degree greater than $d_u$, then if 
\be
m_B &<& d_v + n_j + 2 + \delta_B ,   \label{mBkappa}  \\
m_g - n_g &\le& d_v - n_j ,        \label{mgcond}  \\
d_v &\ge& n_j + 2            \label{dvnjcond}
\ee
then the degree sequence is forcing, where 
$j = \{ w : d_w = d_v - 1 \} = \alpha_v$ and 
$g = \{ w : d_w < d_u\}$.
{\em Proof}. By lemma \ref{unique},
if there is no completable vertex then $u$
must be adjacent to $v$, and therefore $v$ may not be adjacent to all the vertices
of degree less than $d_u$. There are $n - n_j - n_B$ such vertices. It follows that 
\be
\sigma(v,j) &\ge& d_v - (n-n_j-n_B-1) - (n_B-1)  \nonumber \\
&=& d_v - n + n_j +2.            
\ee
Let $t$ be the number of vertices in $j$ that are not
adjacent to all the other vertices in $j$. First consider the case $t=0$.
The conditions are such that the vertices in $j$ are
good, and any of them that are adjacent to both $v$ and all the others in $j$
require an edge to a vertex in $B - v$, so they have two bad edges. Therefore
when $t=0$ we have $\sigma(B,j) \ge 2 \sigma(v,j) + (n_j - \sigma(v,j))
= \sigma(v,j) + n_j$.
Substituting this into (\ref{mBbound}) gives (\ref{mBkappa}).
Next suppose $t>0$. $t$ cannot be equal to 1, so the smallest non-zero value is $t=2$.
If neither of these vertices are adjacent to $v$, then the count of bad edges is unchanged
and we obtain (\ref{mBkappa}) again. If one of them is adjacent to $v$, then the
only way it can avoid having another bad edge is if it has $d_v-2$
good neighbours. Therefore it has $(d_v-2)-(n_j-2) = d_v - n_j$ good neighbours
not in $j$. Under the condition (\ref{mgcond}) this is either not
possible, or else it is possible and no other vertex in $j$ can be adjacent to a vertex
in $g$. In the latter case, the second vertex in $j$ that is not adjacent to all others
in $j$ can have at most $n_j-2$ good neighbours, so it has $d_v - n_j + 1$ bad
neighbours. In this case we obtain 
$\sigma(B,j) \ge 2(\sigma(v,j)-2) + (n_j-\sigma(v,j)) + 1 + (d_v - n_j + 1)
= \sigma(v,j) + d_v -2$. Using (\ref{dvnjcond}) we find
that $\sigma(B,j)$ is no smaller than before, and the condition
(\ref{mBkappa}) applies again. By similar arguments, $t > 2$ also does
not lead to a smaller $\sigma(B,j)$. \qed

\section{Statistics and extension of the concepts}

\begin{table}
\begin{tabular}{ccccccccc}
$n$ & seq & f-ds & graphs & ds-r & forced & weakly ds-r  & good d & good G \\
\hline
1 & 1       & 1         & 1    &  1     &  1  &    1   & 1 & 1 \\
2 & 2       & 2         & 2    &  2     &  2  &    2   & 2 & 2 \\
\hline
3 & 4       & 4         &   4  &  4     &  4  &    4   & 2 & 2\\
4 & 11     & 10      &  11  &  10   &  10     &   11   & 6 & 6 \\
5 & 31     &  29     & 34     &  32  &  30    &    34  & 9 & 9 \\
6 & 102   &  88     &  156   & 128  &  106  &  152  &  30 & 34  \\
7 & 342   & 304     & 1044   & 768  &  616  &  930  &  54 & 76  \\
8 & 1213 & 1034   & 12346  & 7311 & 4385  &  9077   & 183 & 542 \\
9 & 4361 & 3697   & 274668 & 126152  &  60259 & 147812  & 379 & 2731  \\
10 & 16016 & 13158  & 12005168 & 4045030 & 1194438 & 4472390  & 1256 & 71984
\end{tabular}
\caption{Data for small graphs. The columns are: $n$=number of vertices; seq = number of graphic sequences;
f-ds = number of forcing degree sequences;
graphs=number of graphs; ds-r = number of ds-reconstructible graphs; 
forced = number of graphs whose degree sequence is forcing; weakly dsr = number of graphs uniquely
specified by one card and the degree sequence; good  d = number of sequences having all good degrees;
good G = number of graphs having all good vertices.
}
\label{t.data}
\end{table}

Table \ref{t.data} lists information about the number of ds-reconstructible graphs, the number
of ds-reconstruction-forcing degree sequences, and the number of graphs whose degree sequence is forcing
(the latter graphs are a subset of the ds-reconstructible graphs), for graphs on up to 10 vertices.
The numerical evidence is that the proportion of all graphs of given order $n$
that are ds-reconstructible is a decreasing function of $n$, and therefore the fraction of graphs
of large $n$ that are ds-reconstructible is small and may tend to zero as $n \rightarrow \infty$.
This does not rule out that these graphs may contribute a useful part of a proof 
of the reconstruction conjecture (if one is possible) based on finding a collection of
reconstruction methods which together cover all graphs. 

We have defined the notion of ds-reconstructibility so that it is simple and self-contained.
It is not defined in such a way as to capture as wide a class of graphs as possible, but
rather a readily-defined and -recognized class of graphs. For this reason it does not include
some graphs whose reconstruction is easy, such as the path graphs on more than three vertices.

There are several natural extensions to the idea of ds-reconstructibility that include path graphs
and some other graphs. For example, there is a class of graphs such that each graph in the class
contains a card $\Gv$, such that, up to isomorphism, there is only one graph having both the same 
degree sequence as $G$ and having $\Gv$ as a subgraph. 
Let us call such graphs `weakly ds-reconstructible'. This is a larger class than the
ds-reconstructible graphs (as we have defined them) because it includes cases where
degree information in $\Gv$ allows for more than one assignment of neighbours of $v$, but 
all the assignments lead to the same graph, owing to automorphism of $\Gv$. The path
graphs are an example of this.

\begin{figure}
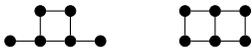

\basicgraph{0.4}
{{(0,0)/a}, {(1,0)/b}, {(2,0)/c}, {(3,0)/d}, {(1,1)/e}, {(2,1)/f}}
{a/b, b/c, b/e, c/d, c/f, e/f} 
\hspace{5mm}
\basicgraph{0.4}
{{(0,0)/a}, {(1,0)/b}, {(2,0)/c}, {(0,1)/d}, {(1,1)/e}, {(2,1)/f}}
{a/b, a/d, b/c, b/e, c/f, d/e, e/f} 
\caption{These two graphs, and their complements, are the smallest graphs that are not
weakly ds-reconstructible (i.e. reconstructible from one card and the degree sequence).}
\label{fig.6}
\end{figure}

The smallest graphs that are not weakly ds-reconstructible have six vertices; they are shown in figure
\ref{fig.6}.

Unigraphic degree sequences produce graphs that are obviously reconstructible from the degree sequence.

It can happen that, for some degree sequence,
(e.g. [11222]), all but one of the graphs of that degree sequence are ds-completable.
But in that case we can reconstruct all the graphs of that degree sequence, because the presence
of a completable card is recognisable. If such a card is present then it is used to reconstruct
the graph. If no such card is found then we know that the graph is the unique graph of the
given degree sequence that has no completable vertex.

Another extension is to ask, what further information would be required to make a card
uniquely completable, when degree information alone is not sufficient? One might for example
consider the class of graphs uniquely specified by a pair of cards and the degree sequence,
or by a single card and further information known to be recoverable from the deck.

The large number of equations in this paper is a symptom of the fact that I have not found
any single simply-stated criterion sufficient to fully distinguish the classes of forcing
and non-forcing degree sequences. A further avenue for investigation would be to
seek a way to manipulate a ds-reconstructible graph so as to obtain 
a non-ds-reconstructible graph of the same degree sequence, if one exists. Another
avenue is to try to reduce the problem to a small number of canonical forms.

\bibliographystyle{plain}
\bibliography{graphrefs}

\end{document}